\documentclass[11pt]{amsart}
\usepackage{amsaddr}
\usepackage[utf8]{inputenc}
\usepackage[english]{babel}
\usepackage{cite}
\usepackage{amsthm,amsmath,amssymb,hyperref,amsfonts,tikz,adjustbox}
\usepackage{mathrsfs}
\usetikzlibrary{cd}
\numberwithin{equation}{section}
\usepackage[shortlabels]{enumitem}
\usepackage[a4paper, total={6.8in, 8in}]{geometry}
\usepackage{imakeidx}

\usepackage{tikz}
\usetikzlibrary{cd}
\usetikzlibrary{calc} 
\usetikzlibrary{decorations.pathmorphing} 

\tikzset{curve/.style={settings={#1},to path={(\tikztostart)
    .. controls ($(\tikztostart)!\pv{pos}!(\tikztotarget)!\pv{height}!270:(\tikztotarget)$)
    and ($(\tikztostart)!1-\pv{pos}!(\tikztotarget)!\pv{height}!270:(\tikztotarget)$)
    .. (\tikztotarget)\tikztonodes}},
    settings/.code={\tikzset{quiver/.cd,#1}
        \def\pv##1{\pgfkeysvalueof{/tikz/quiver/##1}}},
    quiver/.cd,pos/.initial=0.35,height/.initial=0}

\tikzset{tail reversed/.code={\pgfsetarrowsstart{tikzcd to}}}
\tikzset{2tail/.code={\pgfsetarrowsstart{Implies[reversed]}}}
\tikzset{2tail reversed/.code={\pgfsetarrowsstart{Implies}}}
\tikzset{no body/.style={/tikz/dash pattern=on 0 off 1mm}}

\newtheorem{theorem}{Theorem}[section]
\newtheorem{corollary}[theorem]{Corollary}
\newtheorem{lemma}[theorem]{Lemma}
\newtheorem{proposition}[theorem]{Proposition}

\theoremstyle{definition}
\newtheorem{definition}[theorem]{Definition}
\newtheorem{example}[theorem]{Example}

\theoremstyle{remark}
\newtheorem{remark}[theorem]{Remark}

\newtheorem*{problem}{\textbf{Main Problem}}

\newcommand*{\im}{\operatorname{im}}
\newcommand*{\dom}{\operatorname{dom}}
\renewcommand*{\hom}{\operatorname{Hom}}

\title{On the homology of partial group representations}



\author{Emmanuel Jerez}
\address[Emmanuel Jerez]{Departamento de Matem{\'a}tica, Universidade de S\~ao Paulo, Brazil}
\email{ars.ejerez@icloud.com}

\begin{document}

\begin{abstract}
    We study how the partial group (co)homology of a group \( G \) with coefficient in a partial representation \(M\) can be described using the usual group (co)homology.
    To address this, we introduce the concept of the \textit{universal globalization} \( \Lambda(M) \) of a partial group representation \(M\) of \(G\).
    Our main result shows that the partial group homology \( H^{\text{par}}_{\bullet}(G, M) \) is naturally isomorphic to the classical group homology \( H_{\bullet}(G, \Lambda(M)) \).
    We extend this result to the cohomological framework, obtaining a spectral sequence involving the classical group cohomology that converges to the partial group cohomology. Notably, when \( G \) is countable, the spectral sequence collapses, resulting in a natural isomorphism \( H^{\bullet}_{\text{par}}(G, M) \cong H^{\bullet}(G, \operatorname{Hom}_{K_{\text{par}} G}(\Lambda(K_{par}G), M))\),
    where \(K_{par}G\) stands for the partial group algebra of \(G\).
\end{abstract}

\maketitle

\tableofcontents

\section*{Introduction}

Partial group actions, introduced by Exel in \cite{exe}, generalize classical group actions by encoding symmetries defined only on subsets of a space. These arise naturally from the following question: What happens to a family of symmetries when examining their local behavior? For instance, restricting a global action of a group \(G\) on a space \(X\) to a subspace \(Y\) typically results not in a group action, but in a partial one.

Beyond their significant applications in the context of \(C^*\)-algebras, partial group actions have stimulated rich investigations in purely algebraic settings (see, for example, the surveys \cite{EBSurvery}, \cite{D3}). As the theory of partial group actions has developed extensively, a fundamental question emerges: How can one define a proper (co)homology theory for partial group actions?

Initial approaches to partial (co)homology include a semigroup-based cohomology tied to the partial Schur multiplier \cite{DKh, DoNoPi} and a cohomology theory constructed from partial representations \cite{AlAlRePartialCohomology}. The latter is deeply connected to Hochschild (co)homology computations for partial crossed products and partial twisted group algebras \cite{AlAlRePartialCohomology}, \cite{Article_Dokuchaev-Jerez_2023_TTPGAAOPCP}, and it is this (co)homological framework that we study in the present work.

Developing methods to compute the partial group (co)homology of a group \(G\) is crucial not only for understanding the structure of partial group actions and partial representations but also for calculating the Hochschild (co)homology of partial twisted group algebras. In this context, we address the following central problem:
\begin{problem}
    Can the classical group (co)homology of a group \(G\) be used to compute the partial group (co)homology of \(G\) with coefficients in a partial representation? If so, what is the relationship between these two homological theories?
\end{problem}
This paper resolves this problem by establishing a natural isomorphism between classical group (co)homology and partial group (co)homology.
Our results provide a bridge between classical and partial homological invariants, enabling the use of well-established tools from group cohomology to analyze partial structures.
Conversely, partial group (co)homology can be applied to compute the classical group (co)homology of \(G\).
This approach is computationally advantageous because the \textit{canonical} complex used to compute partial group (co)homology is \textit{smaller} than the canonical complex used to compute classical group (co)homology.

In addition to the above, the results obtained in this work include group algebras in the class of Hopf algebras for which \cite[Proposition 5.2]{ABV2} is applicable.  
This is particularly important, as the authors noted in \cite[Remark 5.3]{ABV2} the difficulty of providing examples or counterexamples of Hopf algebras for which the mentioned proposition is applicable.  
We believe that the techniques developed in this work can be extended to the general case of partial representations of Hopf algebras and used to construct a large class of Hopf algebras for which the mentioned proposition is applicable.  
Another significant consequence of the results obtained in this work is a positive answer to an open conjecture \cite[Conjecture D]{MMAMDDHKPartialHomology} in the case where \(G\) is countable.
This conjecture involves a new homological invariant of \(G\), referred to as the \textit{partial projective dimension} of \(G\)  

A central result of this work is the construction of a universal global action for any partial action of \( G \) on \( M \).
This globalization, defined as \(\Lambda(M)\), enables us to transition between partial and global representation of \(G\) within the homological framework.

The first section is dedicated to recalling the basic theory and definitions necessary for the development of this work.  
For the basic theory of partial group actions, we follow Exel's book \cite{E6}. For homological algebra we refer to the books \cite{loday2013cyclic}, \cite{RotmanAnInToHoAl}, and \cite{weibel_1994}.

The second section is dedicated to partial group actions on modules.
We explore the relationship between partial group actions on modules and partial representations of groups, and then we introduce the construction of the \textit{partial} tensor product of two partial group actions on modules. This partial tensor product will play a central role in the results obtained in this work. We prove that this partial tensor product satisfies the usual properties of tensor products on modules, such as certain universal property and \textit{associativity} under certain conditions.
An important consequence of the existence of the partial tensor product is the existence of a universal global action for a partial group action on a $K$-module.
This universal global action may not be a globalization in the usual sense.
In contrast to set-theoretical partial group actions, partial group actions on modules may not be globalizable, and if globalization exists, it may not be unique.
We give necessary and sufficient conditions for a partial group action on modules to be globalizable. Consequently, partial group actions arising from partial representations are always globalizable. We obtain, in this way, a functor \(\Lambda\) from partial representations to global representations that sends a partial representation to its universal globalization. We call this functor the globalization functor.

The third section is the main part of this work.
In this section we prove that the partial group homology \( H^{\text{par}}_{\bullet}(G, M) \) is naturally isomorphic to the classical group homology \( H_{\bullet}(G, \Lambda(M)) \) of its universal globalization. This result extends naturally into the cohomological context through a spectral sequence converging to the partial group cohomology \( H^{\bullet}_{\text{par}}(G, M) \). In particular, for countable groups or specific cases where this spectral sequence collapses, we obtain an  isomorphism \( H^{\bullet}_{\text{par}}(G, M) \cong H^{\bullet}(G, \operatorname{Hom}_{K_{\text{par}} G}(\Lambda(K_{par}G), M)) \).

\section{Preliminaries}

Let $K$ be a commutative unital ring. 
Throughout this work, we fix $K$ as our ground ring. Thus, we refer to $K$-modules simply as modules, and we denote the tensor product of two $K$-modules $X \otimes_{K} Y$ simply as $X \otimes Y$. 
If $\mathcal{A}$ is a $K$-algebra and $X$ is an $\mathcal{A}$-bimodule, the $\mathcal{A}$-tensor product of $n$ copies of $X$ will be denoted by $X^{\otimes_{\mathcal{A}}n}$. When $\mathcal{A} = K$, it will be simply denoted as $X^{\otimes n}$.
The capital letter $G$ will be used to denote a group, and the unit element of an algebraic structure will be denoted simply by $1$ if there is no ambiguity. We use the notation $\mathbb{N}$ to denote the set of natural numbers, $\{0, 1, 2, \ldots\}$. This work is dedicated to the study of partial actions and partial representations of groups; hence, we will use the terms ``\textit{partial action}'' and ``\textit{partial representation}'' to refer to this specific type of partial action and partial representation, respectively.

\subsection{Partial actions and partial representations}

Here we introduce the basic theory of partial group actions and partial representations of groups used in this work, following the book \cite{exe}.

\begin{definition} \label{d: partial representation}
    A \textbf{partial representation} of $G$ in a unital associtive $K$-algebra $A$ is a map $ \pi : G \rightarrow A$ such that, for any $s,t \in G$, we have:
    \begin{enumerate}[(a)]
    \itemsep0em 
        \item $ \pi(s)  \pi(t) \pi(t^{-1} )= \pi(st) \pi(t^{-1}),$ 
        \item $ \pi(s^{-1})  \pi(s) \pi(t)= \pi(s^{-1}) \pi(st),$
        \item $ \pi(1_G)=1_A$.
    \end{enumerate}
    If \(M\) is a \(K\)-module and \(A = \operatorname{End}_{K}(M)\) we say that \(\pi\) is a partial representation of \(G\) on \(M\).
\end{definition}

\begin{definition}
    Let $\pi: G \rightarrow \operatorname{End}_K(M)$ and $\pi': G \rightarrow \operatorname{End}_K(W)$ be two partial representations of $G$. A \textbf{morphism of partial representations} \index{Morphism of partial representations} is a morphism of $K$-modules $f: M \rightarrow W$, such that $f \circ \pi (g) = \pi'(g) \circ f$ $\forall g \in G$.
\end{definition}

The category of partial representations of $G$, denoted $\textbf{ParRep}_G$, is the category whose objects are pairs $(M, \pi)$, where $M$ is a $K$-module and $\pi : G \rightarrow \operatorname{End}_K(M)$ is a partial representation of $G$ on $M$, and whose morphisms are morphisms of partial representations. Denote by $\mathcal{S}(G)$ the Exel's semigroup of $G$, i.e., it is the inverse monoid defined by the generators $[t]$, $t \in G$, and relations: 
    \begin{enumerate}[(i)]
        \item $ [1_G] = 1 $;
        \item $[s^{-1}][s][t] = [s^{-1}][st]$;
        \item $[s][t][t^{-1}] = [st][t^{-1}]$;
    \end{enumerate}
for any $t,s \in G$. We denote the inverse (in the inverse semigroup sense) of an element $w \in \mathcal{S}(G)$ as $w^{*}$. The set of idempotent elements of $\mathcal{S}(G)$ will be denoted by $E(\mathcal{S}(G))$. For more details, refer to \cite{E6}.
\begin{definition}
    We define the \textbf{partial group algebra} $K_{par}G$ as the semigroup $K$-algebra generated by $\mathcal{S}(G)$.  We set $\mathcal{B}$ as the semigroup $K$-algebra generated by $E(\mathcal{S}(G))$.
\end{definition}

Here, we recall some well-known computation rules for \(K_{par}G\) that will be heavily used throughout this work. We will omit the proofs since they can be easily found in the literature (for example, refer to \cite{E6}).

\begin{proposition} \label{p: computations rules}
    Define $e_g := [g][g^{-1}]$, then 
    \begin{enumerate}[(i)]
        \item $[g]e_h = e_{gh}[g]$ for all $g, h \in G$,
        \item $e_g e_h = e_h e_g$ for all $g, h \in G$
        \item $e_g e_g = e_g$ for all $g \in G$
        \item the set $\{ e_g : g \in G \}$ generates the semigroup $E(\mathcal{S}(G))$
    \end{enumerate}
\end{proposition}

Recall that the subalgebra $\mathcal{B} \subseteq K_{par}G$ is a left $K_{par}G$-module with the action
\[
    [g] \triangleright u = [g] u [g^{-1}]
\]
and a right $K_{par}G$-module with the action
\[
    u \triangleleft [g] =  [g^{-1}] u [g].
\]
One verifies that
\begin{enumerate}[(i)]
    \item $w \triangleright u = wu$ for all $w, u \in \mathcal{B}$,
    \item $u \triangleleft w = wu$ for all $w, u \in \mathcal{B}$,
    \item $[g] \triangleright 1 = e_{g}$ for all $g \in G$,
    \item $1 \triangleleft [g] = e_{g^{-1}}$ for all $g \in G$.
\end{enumerate}

Hence, we obtain a well-defined morphism of right $K_{par}G$-modules
\begin{align} \label{eq: epsilon map}
    \varepsilon: K_{par}G   &\to \mathcal{B} \\
                    z       &\to 1 \triangleleft z. \notag
\end{align}
\begin{lemma} \label{l: varepsilon formulas}
    Let $\varepsilon$ be the map \eqref{eq: epsilon map}, then
    \begin{enumerate}[(i)]
        \item $\varepsilon([g_1][g_2] \ldots [g_n]) =  [g_n^{-1}] \ldots [g_2^{-1}] [g_1^{-1}][g_1][g_2] \ldots [g_n]= e_{g_{n}^{-1}} e_{g_{n}^{-1} g_{n-1}^{-1}} \ldots e_{g_{n}^{-1} g_{n-1}^{-1} \ldots g_2^{-1} g_1^{-1}}$ for all $g_1, \ldots, g_n \in G$;
        \item $\varepsilon(e_{g}) = e_{g}$ for all $g \in G$, and consequently $\varepsilon(w) = w$ for all $w \in \mathcal{B}$;
        \item $[h^{-1}]\varepsilon(z) = \varepsilon(z[h])[h^{-1}]$ for all $h \in G$.    \end{enumerate}
\end{lemma}
\begin{proof}
    Items $(i)$ and $(ii)$ are direct consequences of Proposition~\ref{p: computations rules}. For $(iii)$ observe that
    \begin{align*}
        [h^{-1}]\varepsilon([g_1][g_2] \ldots [g_n])
        &= [h^{-1}]e_{g_{n}^{-1}} e_{g_{n}^{-1} g_{n-1}^{-1}} \ldots e_{g_{n}^{-1} g_{n-1}^{-1} \ldots g_2^{-1} g_1^{-1}} \\
        (\text{By Proposition~\ref{p: computations rules}})&= e_{h^{-1}g_{n}^{-1}} e_{h^{-1}g_{n}^{-1} g_{n-1}^{-1}} \ldots e_{h^{-1}g_{n}^{-1} g_{n-1}^{-1} \ldots g_2^{-1} g_1^{-1}}e_{h^{-1}}[h^{-1}] \\
        &= e_{h^{-1}}e_{h^{-1}g_{n}^{-1}} e_{h^{-1}g_{n}^{-1} g_{n-1}^{-1}} \ldots e_{h^{-1}g_{n}^{-1} g_{n-1}^{-1} \ldots g_2^{-1} g_1^{-1}}[h^{-1}] \\
        &= \varepsilon([g_1]\ldots[g_n][h])[h^{-1}].
    \end{align*}
\end{proof}

\begin{proposition}[Proposition 2.5 \cite{exe}] \label{p: basic decomposition of elements of SG}
    Let $G$ be a group, then any element of $z \in \mathcal{S}(G)$ admits a decomposition
    \[
        z = e_{g_1} e_{g_2} \ldots e_{g_n} [h],
    \]
    where $n \geq 0$ and $g_1, \ldots, g_n \in G$. In addition, one can assume that
    \begin{enumerate}[(i)]
        \item $g_i \neq g_j$, for $i \neq j$,
        \item $g_i \neq s$ for any $i$.
    \end{enumerate}
    Furthermore, this decomposition is unique up to the order of the $\{ g_i \}$.
\end{proposition}

\begin{proposition}[Proposition 10.5 \cite{E6}] \label{p: KparG universal property} 
    The map
    \[
       g \in G \mapsto [g] \in  K_{par}G
    \]
    is a partial representation, which we will call the \textbf{universal partial representation} of $G$. In addition, for any partial representation $\pi$ of $G$ in a unital $K$-algebra $A$ there exists a unique algebra homomorphism $\phi : K_{par}G \rightarrow A$, such that $\pi(g) = \phi([g])$, for any $g \in G.$
\end{proposition}

A direct consequence of Proposition~\ref{p: KparG universal property} is the following well-known theorem:

\begin{theorem} \label{t: partial representations and KparG-modules isomorphism}
    The categories $\textbf{ParRep}_G$ and $K_{par}G$-\textbf{Mod} are isomorphic.
\end{theorem}

The key constructions in Theorem~\ref{t: partial representations and KparG-modules isomorphism} are as follows: if $\pi: G \to \operatorname{End}_K(M)$ is a partial group representation, then $M$ becomes a $K_{par}G$-module with the action defined as $[g] \cdot m := \pi_g(m)$. Conversely, if $M$ is a $K_{par}G$-module, we obtain a partial representation $\pi: G \to \operatorname{End}_K(M)$ such that $\pi_g(m):= [g] \cdot m$. Keeping in mind Theorem~\ref{t: partial representations and KparG-modules isomorphism}, we use the terms \textit{partial group representation} and $K_{par}G$\textit{-module} interchangeably.
 \begin{definition} \label{d: partial group action}
     A (left) partial group action $\theta = \big( G, X, \{ X_{g}\}_{g \in G}, \{ \theta_{g} \}_{g \in G} \big)$ of a group $G$ on a set $X$ consists of a family indexed by $G$ of subsets $X_{g} \subseteq X$ and a family of bijections $\theta_{g}: X_{g^{-1}} \to X_{g}$ for each $g \in G$, satisfying the following conditions:
     \begin{enumerate}[(i)]
        \item $X_{1} = X$ and $\theta_{1} = 1_{X}$,
        \item $\theta_{g}(X_{g^{-1}}\cap X_{g^{-1}h}) \subseteq X_{g} \cap X_{h}$,
        \item $\theta_{g}\theta_{h}(x) = \theta_{gh}(x)$, for each $x \in X_{h^{-1}} \cap X_{h^{-1}g^{-1}}$.
     \end{enumerate}
 \end{definition}
Let $X$ be a set, and let $A$, $B$, $D$, $C$ be subsets of $X$. Suppose $\psi: A \to B$ and $\phi: D \to C$ are bijective functions. The partial composition of $\phi \circ \psi$ is defined as the function whose domain is the larger set where $\phi(\psi(x))$ makes sense, i.e., $\text{dom}(\phi \circ \psi) = \psi^{-1}(A \cap C)$. It is worth mentioning that conditions $(ii)$ and $(iii)$ are equivalent to $\theta_{g} \theta_{h}$ being a restriction of $\theta_{gh}$ for all $g, h \in G$. \begin{definition} \label{d: partial action map}
   Let $\alpha = (G, X, \{ X_g \}, \{ \alpha_g \}) $ and $\beta = (S, Y, \{ Y_h \},\{ \beta _{h}\}) $ be partial actions, a morphism of partial actions $\phi$ is a pair $(\phi_0, \phi_1)$ where $\phi_0: X \to Y$ is a map of sets and $\phi_1: G \to S$ is a morphism of groups such that
   \begin{enumerate}[(i)]
   	\item $ \phi_0(X_g) \subseteq  Y_{\phi_1(g)}$,
	\item $ \phi_0(\alpha_g(x)) = \beta_{\phi_1(g)}(\phi_0(x)) $, for all $x \in X_{g^{-1}}$.
   \end{enumerate}
 \end{definition}

  In Definition~\ref{d: partial action map}, if we have $G = S$ and $\phi_0 = \operatorname{id}_{G}$, then we obtain the classical definition of a $G$-equivariant map for partial group actions of a group $G$.
 \begin{definition}
  Let $\alpha$ be a partial group action of $G$ on $X$ and $\beta$ a partial group action of $G$ on $Y$. A $G$-\textbf{equivariant map} $\phi: X \to Y$ is a function such that:
     \begin{enumerate}[a)]
         \item $\phi(X_g) \subseteq Y_g$ for all $g \in G$,
         \item $\phi(\alpha_g(x)) = \beta_g(\phi(x))$ for all $g \in G$ and $x \in X_{g^{-1}}$.
     \end{enumerate}
 \end{definition}
The category of partial group actions of a fixed group $G$ on sets with $G$-equivariant maps will be denoted by $G_{par}$-\textbf{Set}.

\subsection{Homology of partial group representations} \label{s: homological algebra}

We use Weibel's book \cite{weibel_1994} and Rotman's book \cite{RotmanAnInToHoAl} as our primary references for homological algebra theory. The group homology of a $G$-module $M$ is based on the space of coinvariants $M_G := M / DM$, where $DM$ is the submodule of $M$ generated by $\{ g \cdot m - m : m \in M, \, g \in G \}$. In this case, one can verify that $M_{G} \cong K \otimes_{KG} M$. Hence, the functor $(-)_{G}: G\textbf{-Mod} \to K\textbf{-Mod}$ is a right exact functor, and we can define the homology of $G$ with coefficients in a $G$-module $M$ as the left derived functor of the functor of coinvariants. Thus, we define
\[
    H_\bullet(G, M) := \operatorname{Tor}^{KG}_{\bullet}(K, M).
\]
The standard resolution $C''_\bullet \to K$ of right $KG$-modules 
\begin{equation} \label{eq: standar resolution of K}
   \cdots \to KG^{\otimes n + 1} \to KG^{\otimes n + 1} \to \cdots KG^{\otimes 2} \to KG \overset{\hat{\varepsilon}}{\to} K
\end{equation}
is given by $C''_n = KG^{\otimes n + 1}$, $\hat{\varepsilon}: g \in KG \mapsto 1 \in K$, and with differential determined by the face maps
    \[
        d_{i}(g_1 \otimes g_2 \otimes \ldots \otimes g_n \otimes g_{n+1}) :=   
        \left\{\begin{matrix}
            g_2 \otimes \ldots \otimes g_n \otimes g_{n+1} & \text{ if } i = 0;\\ 
            g_1 \otimes \ldots \otimes g_i g_{i+1} \otimes \ldots \otimes g_{n+1} & \text{ if } 0 < i \leq n. 
        \end{matrix}\right.
    \]
Thus, if $M$ is a $KG$-module, the canonical complex that computes the group homology of $M$ with coefficients in $M$ is $C''_\bullet \otimes_{KG} M \cong (C_n(G, M), d)$, such that $C_n(G, M) := KG^{\otimes n} \otimes M$ and $d$ is determined by the face maps
    \[
        d_{i}(g_1 \otimes g_2 \otimes \ldots \otimes g_n \otimes m) :=   
        \left\{\begin{matrix}
            g_2 \otimes \ldots \otimes g_n \otimes m & \text{ if } i = 0;\\ 
            g_1 \otimes \ldots \otimes g_i g_{i+1} \otimes \ldots \otimes m & \text{ if } 0 < i < n; \\ 
            g_1 \otimes \ldots \otimes g_{n-1} \otimes g_n \cdot m & \text{ if } i = n. 
        \end{matrix}\right.
    \]
For any $K_{par}G$-module $M$ we can define the space of coinvariants as $M_G:= M / DM$, where $DM$ is the $K$-submodule of $M$ generated by $\{ [g] \cdot m - e_{g^{-1}} \cdot m : g \in G, \, m \in M \}$.
Observe that if \(M\) is a representation of \(G\) then this \textit{partial} space of coinvariants is just the usual space of coinvariants $M_G$.
Furthermore, a direct computation shows that $M_G \cong \mathcal{B} \otimes_{K_{par}G} M$ as $K$-modules. Hence, we can define the partial group homology of $G$ with coefficients in a $K_{par}G$-module $M$ as follows:

\begin{definition}[\hspace{-0.005em}\cite{AlAlRePartialCohomology}, \cite{MMAMDDHKPartialHomology}]
    Let $G$ be a group, we define the partial group homology of $G$ with coefficients in a left $K_{par}G$-module $M$ as
    \begin{equation}
        H_\bullet^{par}(G,M) := \operatorname{Tor}_\bullet^{K_{par}}(\mathcal{B}, M).
    \end{equation}
    Dually, we define the partial group cohomology of $G$ with coefficients in a right $K_{par}G$-module $M$ as
    \begin{equation}
        H^\bullet_{par}(G,M) := \operatorname{Ext}^\bullet_{K_{par}}(\mathcal{B}, M).
    \end{equation}
\end{definition}

To compute the partial (co)homology, we can use a projective resolution of right $K_{par}G$-modules of $\mathcal{B}$. 
For the particular case $\mathcal{H}_{par} = K_{par}G$, one can use the projective resolution constructed in \cite[Section 6]{MDEJHopf}.
Hence, we obtain a projective resolution 
\begin{equation} \label{eq: standar resolution of B}
    \cdots \to K_{par}G^{\otimes_{\mathcal{B}}  n + 2 } \to K_{par}G^{\otimes_{\mathcal{B}}n + 1}  \to \cdots K_{par}G^{\otimes_{\mathcal{B}}2} \to K_{par}G \overset{\varepsilon}{\to} \mathcal{B}
\end{equation}
$(C'_\bullet, d) \overset{\varepsilon}{\to} \mathcal{B}$ of right $K_{par}G$-modules, such that $C'_n:= (K_{par}G)^{\otimes_\mathcal{B} n+1}$, with differential determined by
 \[
     d =  \sum_{i = 0}^{n} (-1)^{i}d_i : C'_n \to C'_{n-1}.
 \]
 where the face maps are
\begin{equation} \label{eq: d0}
    d_0(z_0 \otimes_\mathcal{B} z_1 \otimes_\mathcal{B} \ldots \otimes_\mathcal{B} z_n) = z_0^* z_0 z_1 \otimes_\mathcal{B} z_2 \otimes_\mathcal{B} \ldots \otimes_\mathcal{B} z_n,
\end{equation}
\begin{equation} \label{eq: di}
    d_i(z_i \otimes_\mathcal{B} z_1 \otimes_\mathcal{B} \ldots \otimes_\mathcal{B} z_n) = z_0 \otimes_\mathcal{B} \ldots \otimes_\mathcal{B} \ldots \otimes_\mathcal{B} z_{i-1}z_i \otimes_\mathcal{B} \ldots \otimes_\mathcal{B} z_n,\text{ for } 0 < i \leq n,
\end{equation}
where $z_i \in \mathcal{S}(G)$.
We can obtain clearer description of the face maps on the basic $n$-chains $z_0 \otimes_{\mathcal{B}} z_1 \otimes_{\mathcal{B}} \ldots \otimes_{\mathcal{B}} z_n$, where $z_i \in \mathcal{S}(G).$ By Proposition~\ref{p: computations rules} and Proposition~\ref{p: basic decomposition of elements of SG}, there exist $u \in E(\mathcal{S}(G))$ and $g_0, g_1, \ldots, g_n \in G$ such that
 \[
     z_0 \otimes_{\mathcal{B}} \ldots \otimes_{\mathcal{B}} z_n = [g_0] \otimes_{\mathcal{B}} [g_1] \otimes_{\mathcal{B}} \ldots \otimes_{\mathcal{B}} [g_n]u.
 \]
 Hence, the Equations \eqref{eq: d0} and \eqref{eq: di} takes the form 
 \begin{align} \label{eq: d0 right}
     d_0([g_0] \otimes_\mathcal{B} \ldots \otimes_\mathcal{B} [g_n]u) 
     &= e_{g_0^{-1}}[g_1] \otimes_{\mathcal{B}} \ldots \otimes_{\mathcal{B}} [g_n]u \\
     &= [g_1] \otimes_{\mathcal{B}} \ldots \otimes_{\mathcal{B}} [g_n]e_{g_n^{-1} g_{n-1}^{-1} \ldots g_1^{-1}g_0^{-1}}u, \notag
 \end{align}
 \begin{align} \label{eq: di right}
     d_i([g_0] \otimes_\mathcal{B} \ldots \otimes_\mathcal{B} [g_n]u) 
     &= [g_0] \otimes_\mathcal{B} [g_1] \otimes_\mathcal{B} \ldots \otimes_{\mathcal{B}} [g_{i-1}][g_i] \otimes_{\mathcal{B}}\ldots \otimes_\mathcal{B} [g_n]u \\
     &= [g_0] \otimes_\mathcal{B} \ldots \otimes_{\mathcal{B}} [g_{i-1}g_i] \otimes_{\mathcal{B}}\ldots \otimes_\mathcal{B} [g_n]e_{g_i^{-1} g_{i-1}^{-1} \ldots g_1^{-1}g_0^{-1}}u, \notag
 \end{align}
 for $0 < i \leq n$.
 
\section{Partial group actions on modules}

The main objective of this work is to compare partial group (co)homology with traditional group homology. To achieve this, we focus on partial group actions on modules and explore possible globalization theorems for these actions. This allows us to compare the homology of a partial group action with the homology of its globalization.

\begin{definition}
    Let $G$ be a group, $\mathcal{A}$ a $K$-algebra, and $M$ be a left (right) $\mathcal{A}$-module. A (left) partial group action of $G$ on $M$ is a set-theoretical partial action $\theta = (G, M, \{ M_g \}, \{ \theta_g \})$ of $G$ on $M$ such that each domain $M_g$ is a left (right) $\mathcal{A}$-submodule of $M$ and $\theta_g: M_{g^{-1}} \to M_g$ is an isomorphism of left (right) $\mathcal{A}$-modules. 
\end{definition}

We can extend the concept of partial group action morphism and $G$-equivariant map to partial group actions on $\mathcal{A}$-modules simply by requiring the corresponding maps to be $\mathcal{A}$-linear maps. The category whose objects are partial group actions on modules and morphisms of partial group actions on left (right) $\mathcal{A}$-modules will be denoted by $\textbf{PA-}_\mathcal{A}\textbf{Mod}$ ($\textbf{PA-}\textbf{Mod}_{\mathcal{A}}$), and for a fixed group $G$ the respective category whose objects are partial group actions of $G$ on left (right) $\mathcal{A}$-modules and the morphisms are partial $G$-equivariant maps of modules will be denoted by $G_{par}\text{-}_{\mathcal{A}}\textbf{Mod}$ ($G_{par}\textbf{-Mod}_{\mathcal{A}}$). By abuse of notation, if $\mathcal{A}=K$, we denote the category of partial group actions on $K$-modules simply by $\textbf{PA-Mod}$ and the category of partial group action of a group $G$ on $K$-modules by $G_{par}\textbf{-Mod}$. Furthermore, if $\theta$ is a (left) partial group action of $G$ on the module $M$, we will refer to $M$ as a (left) $G_{par}$-module, always keeping in mind the partial action $\theta$.

\begin{example} \label{e: G-modules are Gpar-modules}
    Let $\alpha$ be a group action of $G$ on a $K$-module $M$. Then it is clear that $\alpha$ is a partial group action. We will refer to this kind of partial action as a \textbf{global action} to emphasize its global nature. Furthermore, there exist an obvious embedding functor $G\textbf{-Mod} \to G_{par}\textbf{-Mod}$.
\end{example}

 \begin{example}\label{e: partial action induced by a partial representation}
     Let $\pi: G \to \operatorname{End}_{K}(M)$ be a partial representation.
     Then, for all $g \in G$ we define $M_g := e_g \cdot M$ and $\theta_g := \pi_g|_{M_{g^{-1}}}$.
     Then, $\theta := \big( G, M, \{ M_g \}, \{ \theta_g \} \big)$ is a partial group action of $K$-modules of $G$ on $M$.
     Indeed, it is clear that $M_1 = M$ and $\theta_1 = \operatorname{id}_{M}$.
     Since $[g][g^{-1}] e_g \cdot m = e_g \cdot m$ for all $g \in G$ and $m \in M$, we conclude that $\theta_g: M_{g^{-1}} \to M_g$ is a well-defined $K$-linear isomorphism. Let $m \in M_{g^{-1}} \cap M_{g^{-1}h}$, then:
    \[
        \theta_g(m)= [g] \cdot m = [g]e_{g^{-1}}e_{g^{-1}h} \cdot m = e_g e_h [g] \cdot m \in M_g \cap M_{h}.
    \]
    Let $m \in M_{h^{-1}} \cap M_{h^{-1}g^{-1}}$, then
    \[
        \theta_g \theta_h(m) = [g][h]e_{h^{-1}} e_{h^{-1}g^{-1}} \cdot m = [gh] e_{h^{-1}} e_{h^{-1}g^{-1}} \cdot m = \theta_{gh}(m).
    \]
    Thus, $\theta$ is a partial group action. We call this partial group action the \textbf{induced partial action} of $G$ on the $K_{par}G$-module $M$. Let $\phi: M \to N$ be a morphism of $K_{par}G$-modules. Then $\phi$ is a $G$-equivariant map with respect to the partial group actions \(\alpha\) and \(\beta\) induced by $M$ and $N$, respectively. Indeed, observe that 
    \[
        \phi(M_g)= \phi(e_g \cdot M) = e_g \cdot \phi(M) \subseteq e_g \cdot N = N_g,
    \]
    and
    \[
        \phi(\alpha_g(e_{g^{-1}}\cdot m)) = \phi([g]e_{g^{-1}} \cdot m) = [g] \cdot  \phi(e_{g^{-1}} \cdot m) = \beta_g(\phi(e_{g^{-1}} \cdot m).
    \]
     Hence, we obtain a well-defined functor $\Upsilon: \textbf{ParRep}_G \to G_{par}\textbf{-Mod}$.
 \end{example}

 \begin{remark} \label{r: domain of partial action induced by a partial representation}
     Observe that in Example~\ref{e: partial action induced by a partial representation} we have that $M_g = [g] \cdot M$. Indeed, it is clear that $e_g \cdot M \subseteq [g] \cdot M$ for all $g \in G$. Then, $[g] \cdot M = [g] e_{g^{-1}} \cdot M \subseteq [g][g^{-1}] \cdot M = e_{g} \cdot M$ for all $g \in G$.
 \end{remark}

 \begin{example}
     Let $X = K_{par}G$, then the corresponding partial group action of $G$ on the right $K_{par}G$-module $X$ is given by:
     \begin{enumerate}[(i)]
        \item $X_g := e_{g^{-1}}K_{par}G$
        \item $\theta_g(e_{g^{-1}}z) := [g] e_{g^{-1}}z = [g]z$ for all $z \in K_{par}G$.
     \end{enumerate}
     It is clear that each $X_g$ is a right $K_{par}G$-module with the inherited structure of $K_{par}G$ and that $\theta_g$ is an isomorphism of right $K_{par}G$-modules.
 \end{example}

\begin{proposition} \label{p: set-theoretical partial actions induces module partial actions}
    Let $\theta = (G, X, \{ X_g \}_{g \in G}, \{ \theta_g \}_{g \in G})$ be a set-theoretical partial group action of $G$ on $X$. Then, $\theta$ determines a partial group action $\hat{\theta}:= \big( G, KX, \{ KX_{g} \}_{g \in G}, \{ \hat{\theta}_g \}_{g \in G} \big)$ of $G$ on the free $K$-module $KX$, where $\hat{\theta}_{g}: KX_{g^{-1}} \to KX_g$ is the $K$-linear isomorphism determined by $\theta_g$.
\end{proposition}
\begin{proof}
    Let $z \in KX_{g^{-1}} \cap KX_h$, then $z = \sum_i \lambda_i x_i$, such that $x_i \in X_{g^{-1}} \cap X_h$ for all $i$. Therefore, 
    \[
        \hat{\theta}_g(z) = \sum_i \lambda_i \theta_g(x_i) \in KX_g \cap KX_{gh}.
    \]
    Let $z \in X_{g^{-1}} \cap X_{g^{-1}h^{-1}}$, then $z = \sum_i \lambda_i x_i$, such that $x_i \in X_{g^{-1}} \cap X_{g^{-1}h^{-1}}$ for all $i$. Thus,
    \[
        \hat{\theta}_h\hat{\theta}_g(z) = \sum_i \lambda_i \theta_h \theta_g(x_i) = \sum_i \lambda_i \theta_{hg}(x_i) = \hat{\theta}_{hg}(z).
    \]
\end{proof}

Let $\alpha = (G, X, \{ X_g \}, \{ \alpha_g \})$ and $\beta = (H, Y, \{ Y_g \}, \{ \beta_g \})$ be two partial group actions and $\phi=(f, \varphi): \alpha \to \beta$ a morphism of partial actions. Then, $\phi$ induces a morphism of partial actions $(\hat{f}, \varphi):\hat{\alpha} \to \hat{\beta}$ where $\hat{\alpha}$ and $\hat{\beta}$ are the partial actions on modules determined by $\alpha$ and $\beta$ respectively as in Proposition~\ref{p: set-theoretical partial actions induces module partial actions} and $\hat{f}: KX \to KY$ is the $K$-linear extension of $f$. Therefore, we obtain a functor $\textbf{PA} \to \textbf{PA-Mod}$.

\begin{proposition} \label{p: set-theoretical partial actions determines partial representations}
    Let $\theta = (G, X, \{ X_g \}_{g \in G}, \{ \theta_g \}_{g \in G})$ be a set-theoretical partial group action of $G$ on $X$, and $K$ be a commutative unital ring.
    Then, $\theta$ determines a partial representation $\pi^{\theta} := \pi$ of $G$ on $KX$, such that 
    \[
        \pi_g(x)=\left\{\begin{matrix}
            \theta_g(x) & \text{ if } x \in X_{g^{-1}}, \\ 
            0 & \text{ otherwise.} 
        \end{matrix}\right.
    \]
\end{proposition}
\begin{proof}
    Clearly, $\pi_1 = \operatorname{id}_{KX}$. Let $g,h \in G$ and $x \in X$, then
    \begin{enumerate}[(i)]
        \item if $x \notin X_{h^{-1}}$, then $\pi_{g^{-1}}\pi_{g}\pi_h(x)=  0$, and $\pi_{g^{-1}}\pi_{gh}(x) = 0$ since:
            \begin{enumerate}[(a)]
                \item if $x \notin X_{h^{-1}g^{-1}}$, then
                    \[
                        \pi_{g^{-1}}\pi_{gh}(x)=  0,
                    \]
                \item if $x \in X_{h^{-1}g^{-1}}$, then
                    \[
                        \pi_{g^{-1}}\pi_{gh}(x)= \pi_{g^{-1}}(\theta_{gh}(x)),
                    \]
                    \begin{enumerate}[(1)]
                        \item if $\theta_{gh}(x) \notin X_{g}$, then
                            \[
                                \pi_{g^{-1}}\pi_{gh}(x)= \pi_{g^{-1}}(\theta_{gh}(x)) = 0,
                            \]
                        \item if $\theta_{gh}(x) \in X_{g}$, then 
                            \[
                                \theta_{g^{-1}}\theta_{gh}(x) = \theta_{h}(x),
                            \]
                            therefore $x \in X_{h^{-1}}$, which contradicts the hypotheses of (i). Thus, $\theta_{gh}(x) \notin X_g$.
                    \end{enumerate}
            \end{enumerate}
        \item If $x \in X_{h^{-1}}$, then $\pi_{g^{-1}}\pi_{g}\pi_{h}(x) = \pi_{g^{-1}}\pi_{g}(\theta_h(x))$,
            \begin{enumerate}[(a)]
                \item if $\theta_h(x) \in X_{g^{-1}}$, then
                    \[
                        \pi_{g^{-1}}\pi_{g}\pi_{h}(x) = \pi_{g^{-1}}\theta_{g}(\theta_h(x)) = \pi_{g^{-1}} \theta_{gh}(x) = \pi_{g^{-1}}\pi_{gh}(x),
                    \]
                \item if $\theta_h(x) \notin X_{g^{-1}}$, then $\pi_{g^{-1}}\pi_{g}\pi_{h}(x) = 0$, and
                    \begin{enumerate}[(1)]
                        \item if $x \notin X_{h^{-1}g^{-1}}$, then $\pi_{g^{-1}}\pi_{gh}(x) = 0$,
                        \item if $x \in X_{h^{-1}g^{-1}}$, then $\theta_{gh}(x) \notin X_g$ and consequently $\pi_{g^{-1}} \pi_{gh}(x) = \pi_{g^{-1}}(\theta_{gh}(x))=0$. Indeed, note that if $\theta_{gh}(x) \in X_{g}$, then $\theta_h(x) = \theta_{g^{-1}}\theta_{gh}(x) \in X_{g^{-1}}$, which contradicts (b) of (ii), thus $\theta_{gh}(x) \notin X_{g}$.
                    \end{enumerate}
            \end{enumerate}
    \end{enumerate}
    Henceforth, $\pi_{g^{-1}}\pi_g \pi_h = \pi_{g^{-1}}\pi_{gh}$ for all $g,h \in G$, analogously one proves that $\pi_{h}\pi_{g}\pi_{g^{-1}} = \pi_{hg}\pi_{g^{-1}}$ for all $g,h \in G$.
\end{proof}

\begin{remark}
    Let $\theta$ be a partial group action of a group $G$ on a set $X$, then by Proposition~\ref{p: set-theoretical partial actions determines partial representations} we have a partial group representation $\pi^{\theta}: G \to \operatorname{End}_K(KX)$. Then, by Theorem~\ref{t: partial representations and KparG-modules isomorphism} $KX$ is a $K_{par}G$-module with the left action determined by
    \begin{equation}
        [g] \cdot x =\left\{\begin{matrix}
            \theta_g(x) & \text{ if } x \in X_{g^{-1}}, \\ 
            0 & \text{ otherwise.} 
        \end{matrix}\right.
    \end{equation}
\end{remark}

Proposition~\ref{p: set-theoretical partial actions determines partial representations} allows us to study partial group actions through the partial group representations they determine. However, maps between partial group actions do not necessarily induce morphisms between the corresponding partial representations. In other words, the linearization of set-theoretical partial group actions into partial representations is not functorial.

\begin{example}
    Let $X = \{ x, y \}$ be the set with two points, and let $G= \{ 1, \, g : g^{2}=1 \}$ be the cyclic group of order $2$. Define the partial action $\alpha:= (G, X, \{ D_g \}, \{ \alpha_g \})$ such that $D_g := \{ x \}$ and $\alpha_g(x)=x$, consider $\beta$ as the trivial group action of $G$ on $X$. It is clear that the identity map $\operatorname{id}: X \mapsto X$ is a $G$-equivariant map from $(\alpha, X)$ to $(\beta, X)$. But the $K$-linear map $\operatorname{id}: KX \to KX$ is not a morphism of $K_{par}G$-modules since
    \[
        \operatorname{id}([g] \cdot_{\alpha} y) = \operatorname{id}(0)=0 \text{ and } [g] \cdot_{\beta} \operatorname{id}(y) = [g] \cdot_{\beta} y = y,
    \]
    where $\cdot_{\alpha}$ and $\cdot_{\beta}$ denotes the left actions of $K_{par}G$ on $KX$ induced by $\alpha$ and $\beta$ respectively.
\end{example}

\begin{remark}
    Let $\theta$ be a set-theoretical partial group action, then by Proposition~\ref{p: set-theoretical partial actions induces module partial actions} we have a partial group action $\hat{\theta}$ of $G$ on $KX$, and by Proposition~\ref{p: set-theoretical partial actions determines partial representations} we obtain a partial group representation $\pi^{\theta}$. Thus, it is immediate that the partial action determined by $\pi^{\theta}$ is $\hat{\theta}$.
\end{remark}

Analogously to the definition of left partial group action, we can define a right partial group action. In particular, we will use right partial group actions on $\mathcal{A}$-modules, where $\mathcal{A}$ is an algebra.

\begin{definition}
    A right partial group action $\beta$ of $G$ on a left (right) $\mathcal{A}$-module $X$ is a left partial group action of $G^{op}$ on $X$, explicitly a right partial group action of $G$ on $X$ is a tuple $\beta:=\big(G, X, \{ X_g \}, \{ \beta_{g} \} \big)$, where $\beta_{g}: X_{g^{-1}} \to X_{g}$ is an isomorphism of left (right) $\mathcal{A}$-modules for all $g \in G$, such that
    \begin{enumerate}[(i)]
        \item $X_1 = X$ and $\beta_1 = 1_{X}$;
        \item $(X_{h} \cap X_{g^{-1}})\beta_{g} \subseteq X_{hg} \cap X_g$;
        \item $(x)\beta_g \beta_h = (x)\beta_{gh}$ for all $x \in X_{h^{-1}g^{-1}} \cap X_{g^{-1}}$,
    \end{enumerate}
    where $(x)\beta_{g} := \beta_g(x)$.
\end{definition}

The category of right partial group actions on modules of a group $G$ will be denoted by \textbf{Mod}-$G_{par}$. Thus, by definition we have a natural isomorphism between the categories $G^{op}_{par}$-\textbf{Mod} and \textbf{Mod}-$G_{par}$. Analogously, to the left case, if $\beta$ is a right partial group action on the module $M$ we say that $M$ is a right $G_{par}$-module.

\textbf{Notation.} 
Given a group $G$ and a set/module $X$ we use the notation $\alpha: G \curvearrowright X$ to denote the partial group action of $G$ on $X$, such that $\alpha := (G, X, \{ X_g \}_{g \in G}, \{ \alpha_g \}_{g \in G})$. Analogously, we use the notation $\beta : X \curvearrowleft G$ to denote the right partial group action of $G$ on $X$, such that $\beta := (G, X, \{ X_g \}_{g \in G}, \{ \beta_g \}_{g \in G})$. 

\begin{example}
    Analogously to Example~\ref{e: partial action induced by a partial representation}, if $M$ is a right $K_{par}G$-module, it determines a right partial group action $\theta:= (G, M, \{ M_{g} \}, \{ \theta_g \})$ of $G$ on $M$ such that:
    \begin{enumerate}[(i)]
        \item $M_{g}:= M \cdot e_{g^{-1}}$;
        \item $\theta_{g}: M_{g^{-1}} \to M_{g}$ is such that $(m)\theta_{g}:= m \cdot [g]$, for all $g \in G$ and $m \in M_{g^{-1}}$.
    \end{enumerate}
    Thus, we obtain a functor $\textbf{Mod-}K_{par}G \to \textbf{Mod-}G_{par}$.
\end{example}

\begin{example}
    Let $X = K_{par}G$, then we have a right partial action of left $K_{par}G$-modules defined as follows:
    \begin{enumerate}[(i)]
       \item $X_g := K_{par}G e_{g^{-1}}$
       \item $(z e_{g})\theta_g := z e_{g} [g] = z[g]$ for all $z \in K_{par}G$.
    \end{enumerate}
    It is clear that each $X_g$ is a left $K_{par}G$-module with the inherited structure of $K_{par}G$ and that $\theta_g$ is an isomorphism of left $K_{par}G$-modules.
\end{example}

Partial actions induced by partial representations will be one of the principal objects of study in this work. Thus, here we show some properties of this type of partial group action.

\begin{lemma} \label{l: partial actions induced by partial representations}
    Let $M$ and $N$ be two left $K_{par}G$-modules, and let $f: M \to N$ be a morphism between them. Let $\alpha: G \curvearrowright M$ and $\beta: G \curvearrowright N$ be the respective induced partial action. Then,
    \begin{enumerate}[(i)]
        \item $[g] \cdot m = \alpha_g(e_{g^{-1}} \cdot m)$ for all $g \in G$ and $m \in M$;
        \item $\dom \alpha_{g_1}\alpha_{g_2} \ldots \alpha_{g_n} = \varepsilon([g_1][g_2] \ldots [g]) \cdot M$ for all $g_1, \ldots, g_n \in G$;
        \item Let $y \in N_{g} \cap \im f$, then there exist $x \in M_g$ such that $f(x) = y$.
    \end{enumerate}
\end{lemma}

\begin{proof}
    By the definition of $\alpha$ we know that $e_{g^{-1}} \cdot m \in M_{g^{-1}}$ for all $g \in G$ and $m \in M$. Then, $\alpha_g( e_{g^{-1}} \cdot m) = [g] e_{g^{-1}} \cdot m = [g] \cdot m$. This proves $(i)$. 
    
    By the axioms of partial action we know that $\dom \alpha_g \alpha_h = M_{h^{-1}} \cap M_{h^{-1}g^{-1}} = e_{h^{-1}}e_{h^{-1}g^{-1}} \cdot M$, and $e_{h^{-1}}e_{h^{-1}g^{-1}}= [h^{-1}][g^{-1}][g][h] =\varepsilon([g][h])$. Thus, by induction, suppose that $\dom \alpha_{g_1} \alpha_{g_2} \ldots \alpha_{g_n} = \varepsilon([g_1][g_2] \ldots [g_n]) \cdot M$, then
    \begin{align*}
       \dom \alpha_{g_1} \ldots \alpha_{g_n} \alpha_{g_{n+1}}  
        &= \alpha_{g_{n+1}^{-1}} \big( \dom \alpha_{g_1} \ldots \alpha_{g_n} \cap \im \alpha_{g_{n+1}} \big)\\
        &= \alpha_{g_{n+1}^{-1}} \big(\varepsilon([g_1][g_2] \ldots [g_n]) \cdot M \cap e_{g_{n+1}} \cdot M \big)\\
        &= \alpha_{g_{n+1}^{-1}} \big(e_{g_{n+1}}\varepsilon([g_1][g_2] \ldots [g_n]) \cdot M \big) \\
        &= [g_{n+1}^{-1}]e_{g_{n+1}}\varepsilon([g_1][g_2] \ldots [g_n]) \cdot M \\
        (\text{by Lemma~\ref{l: varepsilon formulas}})&= \varepsilon([g_1][g_2] \ldots [g_n][g_{n+1}]) [g_{n+1}^{-1}]\cdot M \\
        (\text{by Remark~\ref{r: domain of partial action induced by a partial representation}}) &= \varepsilon([g_1][g_2] \ldots [g_n][g_{n+1}]) e_{g_{n+1}^{-1}}\cdot M \\
        &= \varepsilon([g_1][g_2] \ldots [g_n][g_{n+1}]) \cdot M,
    \end{align*}
    what proves $(ii)$. For $(iii)$, observe that if $y \in N_g \cap \text{im} , f$, then $e_g \cdot y = y$, and there exists $m \in M$ such that $f(m) = y$. Hence, $y = e_g \cdot y = e_g \cdot f(m) = f(e_g \cdot m)$. Therefore, setting $x = e_g \cdot m \in M_g$ yields the desired element.
\end{proof}

\underline{From now on}, when dealing with a partial representation on $M$ (or equivalently a $K_{par}G$-module), we will always take into account the partial group action on $M$ inherited from the partial representation. This is crucial to consider, as we will carry out constructions using partial representations and immediately regard them as partial group actions without explicitly mentioning the functor $\Upsilon$.

In the category $G_{par}$-\textbf{Mod}, if $\alpha: G \curvearrowright M$ is a partial group action on the $K$-module $M$, we can define the space of coinvariants of $M$ as $M_{G}:= M / DM$, where $DM$ is the $K$-submodule of $M$ generated by $\{ \alpha_g(m) - m : g \in G, \, m \in M_{g^{-1}} \}$. This is clearly a generalization of the global case. We can consider the trivial group action on $M_G$, then the canonical map $\pi: M \to M_{G}$ is a $G$-equivariant map. Indeed, for all $g \in G$ and $m \in M_{g}$ we have
\[
    \pi_G(\alpha_g(m)) = \overline{\alpha_g(m)} = \overline{m} = g \cdot \overline{m} = g \cdot \pi_G(m).
\]
Furthermore, note that if \(M\) is a \(K_{par}G\)-module then the space of coinvariants of \(M\) as \(K_{par}G\)-module is equal to the space of coinvariants of \(M\) as \(G_{par}\)-module.

\subsection{Partial tensor product}

Since we have introduced the concept of left and right $G_{par}$-modules, it is natural to ask how to define a tensor product version for this kind of partial structures.
Let $\beta: X \curvearrowleft G$ and $\alpha: G \curvearrowright Y$ be partial actions on modules. Define $\mathcal{K}_{\beta, \alpha}$ as the $K$-submodule of $X \otimes_{K} Y$ generated by the set
\[
    \big\{ (x)\beta_g \otimes y - x \otimes \alpha_g(y) : g \in G, \, x \in X_{g^{-1}} \text{ and } y \in Y_{g^{-1}}\big\}.
\]

Hence, we define the \textbf{partial tensor product} of the right $G_{par}$-module $X$ and the left $G_{par}$-module $Y$ as
\begin{equation}
    X \otimes_{G_{par}} Y := \frac{X \otimes_K Y}{\mathcal{K}_{\beta, \alpha}}.
\end{equation}
Notice that we obtain a bilinear map $\otimes_{G_{par}}: X \times Y \to X \otimes_{G_{par}} Y$. Thus, we write $x \otimes_{G_{par}} y$ to denote $\otimes_{G_{par}}(x,y)$.
   
\begin{proposition}
    Let $\beta: X \curvearrowleft G$ and $\alpha: G \curvearrowright Y$ be partial actions on modules, and $Z$ be a $K$-module. Suppose that $\phi: X \times Y \to Z$ is a bilinear map such that $\phi((x)\beta_g, y) = \phi(x, \alpha_g(y))$ for all $g \in G$, $x \in X_{g^{-1}}$, $y \in Y_{g^{-1}}$. Then, there exists a unique $K$-linear map $\tilde{\phi}: X \otimes_{G_{par}} Y \to Z$ such that the following diagram commutes:
    \[\begin{tikzcd}
        {X \times Y} & Z \\
        {X \otimes_{G_{par}} Y}
        \arrow["\phi", from=1-1, to=1-2]
        \arrow["{\otimes_{G_{par}}}"', from=1-1, to=2-1]
        \arrow["{\tilde{\phi}}"', dashed, from=2-1, to=1-2]
    \end{tikzcd}\]
   \label{p: universal property of partial tensor product}
\end{proposition}

\begin{proof}
    Since $\phi$ is bilinear we have a linear map $\phi': X \otimes Y \to Z$, and by the hypotheses we have that $\mathcal{K}_{\beta, \alpha} \subseteq \ker \phi'$, thus such morphism exists. The uniqueness is a consequence of the fact $\tilde{\phi}(x \otimes_{G_{par}} y) = \phi(x, y)$ and that $\{ x \otimes_{G_{par}} y : x \in X \text{ and } y \in Y \}$ generates $X \otimes_{G_{par}} Y$ as $K$-module.
\end{proof}

By a classical proof of uniqueness given by universal properties one proves that $X \otimes_{G_{par}} Y$ is the unique module (up to isomorphism) that satisfies the universal property mentioned in Proposition~\ref{p: universal property of partial tensor product}.

\begin{remark} \label{r: partial tensor prouct of G-modules}
    If $M$ is a right $G$-module and $N$ a left $G$-module then $M \otimes_{G_{par}}N$ is naturally isomorphic to $M \otimes_{KG}N$.
\end{remark}

Let $\nu: M \curvearrowleft G$, $\alpha: G \curvearrowright A$ and $\beta: G \curvearrowright B$ be partial actions on $K$-modules, and let $f: A \to B$ a left $G$-equivariant map. Then, the map $f_*: M \otimes_{G_{par}} A \to M \otimes_{G_{par}} B$ such that $f_*(m \otimes_{G_{par}} a) := m \otimes_{G_{par}} f(a)$ is a well-defined map of $K$-modules. Henceforth, we obtain a functor
\begin{equation} \label{eq: partial tensor product functor i}
    M \otimes_{G_{par}} - : G_{par}\textbf{-Mod} \to K\textbf{-Mod},
\end{equation}
analogously, if $\nu: G \curvearrowright M$ is a partial action on $M$, we obtain a functor 
\begin{equation} \label{eq: partial tensor product functor ii}
    - \otimes_{G_{par}} M : \textbf{Mod-}G_{par} \to K\textbf{-Mod},
\end{equation}

\begin{remark} \label{r: partial tensor product for partial representations}
    Let $\beta: M \curvearrowleft G$ be a partial action on the module $M$. By abuse of notation we define the functor
    \[
       M \otimes_{G_{par}} - : K_{par}G \textbf{-Mod} \to K \textbf{-Mod}
    \]
    as the composition of the functor $\Upsilon$ and the functor \eqref{eq: partial tensor product functor i}. That is, if $X$ is a left $K_{par}G$-module, we consider the induced partial action $\theta: G \curvearrowright X$ of $G$ on $X$, as in Example~\ref{e: partial action induced by a partial representation}. Thus, $M \otimes_{G_{par}} X$ is the partial tensor product of the right $G_{par}$-module $M$ and the left $G_{par}$-module $X$. In this case, the module $\mathcal{K}_{\beta, \theta}$ is the $K$-submodule generated by the set
    \[
       \big\{ (m)\beta_{g} \otimes e_{g^{-1}} \cdot x - m \otimes [g] \cdot x : g \in G, m \in M_{g^{-1}} \text{ and } x \in X  \big\}.
    \]
    Analogously, if $\alpha: G \curvearrowright M$ is a partial action, we obtain a functor
    \[
       - \otimes_{G_{par}} M : \textbf{Mod-}K_{par}G  \to K \textbf{-Mod}.
    \]
\end{remark}

\begin{proposition} \label{p: partial tensor product A-R bimodule}
    Let $\mathcal{A}$ and $\mathcal{R}$ be $K$-algebras, and let $\beta: X \curvearrowleft G$ be a right partial group action of left $\mathcal{A}$-modules, and let $\alpha: G \curvearrowright Y$ be a left partial group action of right $\mathcal{R}$-modules. Then, $X \otimes_{G_{par}} Y$ is an $\mathcal{A}$-$\mathcal{R}$-bimodule.
\end{proposition}

\begin{proof}
   We know that $X \otimes_{K} Y$ is an $\mathcal{A}$-$\mathcal{R}$-bimodule. Observe that for all $a \in \mathcal{A}$ and $b \in \mathcal{R}$ we have that
   \[
       a \big(x \otimes \alpha_{g}(y) - (x)\beta_g \otimes y \big) b 
       = ax \otimes \alpha_g(yb) - (ax)\beta_g \otimes yb,
   \]
    therefore $\mathcal{K}_{\beta, \alpha}$ is an $\mathcal{A}$-$\mathcal{R}$-subbimodule of $X \otimes_K Y$. Whence we conclude that $X \otimes_{G_{par}} Y$ is an $\mathcal{A}$-$\mathcal{R}$-bimodule.
\end{proof}

\begin{proposition} \label{p: partial tensor product is right exact}
    Let $\beta: M \curvearrowleft G$ be a right partial group action on the $K$-module $M$. Then, the functor $M \otimes_{G_{par}} - : K_{par}G \textbf{-Mod} \to K\textbf{-Mod}$ is right exact. Similarly, if $\alpha: G \curvearrowright M$ is a left partial action on $M$, then the functor $- \otimes_{G_{par}} M : \textbf{Mod-}K_{par}G \to K \textbf{-Mod}$ is right exact.
\end{proposition}

\begin{proof}
    Let $A \overset{f}{\to} B \overset{h}{\to} C \to 0$ be an exact sequence in $K_{par}G$-\textbf{Mod}. Then, we obtain a sequence $M \otimes_{G_{par}} A \overset{f_*}{\to} M \otimes_{G_{par}} B \overset{h_*}{\to} M \otimes_{G_{par}} C \to 0$, since $h$ is surjective we obtain that $h_*$ is surjective. 
    Furthermore, since $h_* f_* = 0$, we have that $\im f_* = \ker h_*$ if and only if the linear map 
    \[
        \overline{h_*}: (M \otimes_{G_{par}} B) / \im f_* \to M \otimes_{G_{par}}C    
    \]
    is an isomorphism, where $\overline{h_*}$ is the map induced by $h_*$. 
    For all $c \in C$, choose $b_{c} \in B$ such that $h(b_{c}) = c$, and define the map $\eta: M \otimes C \to M \otimes_{G_{par}} B/ \im f_*$ such that $\eta(m \otimes c) = \overline{m \otimes_{G_{par}} b_c}$ for all $m \in M$ and $c \in C$.
    To verify that this map is well-defined we have to check that $\eta$ does not depend on the choice of $b_c$. 
    Let $b$ such that $h(b)=h(b_c)=c$, then $b - b_c \in \ker h = \im f$. Consequently, $m \otimes_{G_{par}} (b-b_c) \in \im f_{*}$, whence we have that $\eta$ is well-defined.
    
    Now, suppose that $m \in M_{g^{-1}}$ and $c \in C_{g^{-1}} = e_{g^{-1}} \cdot C$. Since $h$ is a morphism of $K_{par}G$-module, by $(iii)$ of Lemma~\ref{l: partial actions induced by partial representations}, there exists $b \in B_{g^{-1}}= e_{g^{-1}} \cdot B$ such that $h(b)=c$. Then,
   \[
       \eta\big((m)\beta_{g} \otimes c \big) = \overline{(m)\beta_{g} \otimes_{G_{par}} b_{c}} = \overline{(m)\beta_{g} \otimes_{G_{par}} b} = \overline{m \otimes_{G_{par}} [g] \cdot b} = \overline{m \otimes_{G_{par}} b_{[g] \cdot c}}= \eta \big( m \otimes [g] \cdot c \big)
   \]
    Henceforth, by Proposition~\ref{p: universal property of partial tensor product}, there exists a $K$-linear map $\overline{\eta}: M \otimes_{G_{par}} C \to M \otimes_{G_{par}} B / \im f_*$ such that $\overline{\eta}(m \otimes_{G_{par}}c) = \overline{m \otimes_{G_{par}} b_c}$. Finally, note that $\overline{h_{*}}^{-1} = \overline{\eta}$. One shows in the same fashion that $- \otimes_{G_{par}} M$ is right exact.
\end{proof}

\begin{proposition} \label{p: partial tensor product is associative}
   Let $\beta: X \curvearrowleft G$ be a partial action, and let $Y$ be a $K_{par}G$-$\mathcal{A}$-module. Then the functors
   \[
       \big(X \otimes_{G_{par}} Y \big) \otimes_{\mathcal{A}} - : \mathcal{A}\textbf{-Mod} \to K \textbf{-Mod}, \text{ and } 
       X \otimes_{G_{par}} \big( Y \otimes_{\mathcal{A}} - \big): \mathcal{A}\textbf{-Mod} \to K \textbf{-Mod}
   \]
   are isomorphic. Analogously, if $\alpha: G \curvearrowright Y$ is a partial action, $X$ is an $\mathcal{A}$-$K_{par}G$-module, then the functors $- \otimes_{\mathcal{A}} \big(X \otimes_{G_{par}} Y \big)$ and $\big( - \otimes_{\mathcal{A}} X \big) \otimes_{G_{par}} Y$ are isomorphic.
\end{proposition}
\begin{proof}
   Let $Z$ be a left $\mathcal{A}$-module. For a fixed $z \in Z$ consider the map
   \begin{align*}
       X \times Y &\to X \otimes_{G_{par}} (Y \otimes_{\mathcal{A}} Z) \\
       (x,y) &\mapsto x \otimes_{G_{par}}(y \otimes_{\mathcal{A}} z)
   \end{align*}
   It is clear that this map is bilinear and that satisfies the conditions of Proposition~\ref{p: universal property of partial tensor product}, thus for all $z \in Z$ we obtain a morphism of $K$-modules
   \begin{align*}
       X \otimes_{G_{par}} Y &\to X \otimes_{G_{par}} (Y \otimes_{\mathcal{A}} Z) \\
       x \otimes_{G_{par}} y &\mapsto x \otimes_{G_{par}}(y \otimes_{\mathcal{A}} z).
   \end{align*}
   Furthermore, the map
   \begin{align*}
       (X \otimes_{G_{par}} Y) \times Z &\to X \otimes_{G_{par}} (Y \otimes_{\mathcal{A}} Z) \\
       (x \otimes_{G_{par}} y, z) &\mapsto x \otimes_{G_{par}}(y \otimes_{\mathcal{A}} z)
   \end{align*}
   is an $\mathcal{A}$-bilinear map, thus we obtain a $K$-linear map
   \begin{align*}
       \eta: (X \otimes_{G_{par}} Y) \otimes_{\mathcal{A}} Z &\to X \otimes_{G_{par}} (Y \otimes_{\mathcal{A}} Z) \\
       (x \otimes_{G_{par}} y) \otimes_{\mathcal{A}} z &\mapsto x \otimes_{G_{par}}(y \otimes_{\mathcal{A}} z).
   \end{align*}
   On the other hand, we have the following commutative diagram given by Proposition~\ref{p: universal property of partial tensor product}
   \[\begin{tikzcd}
       {X \otimes_K Y \otimes_{\mathcal{A}}Z} & {(X \otimes_{G_{par}}Y) \otimes_{\mathcal{A}}Z} \\
       {X \otimes_{G_{par}}(Y \otimes_{\mathcal{A}}Z)}
       \arrow[from=1-1, to=1-2]
       \arrow["{\otimes_{G_{par}}}"', from=1-1, to=2-1]
       \arrow["\psi"', curve={height=12pt}, dashed, from=2-1, to=1-2]
   \end{tikzcd}\]
    where the horizontal morphism is given by $x \otimes_K y \otimes_{\mathcal{A}} z \mapsto (x \otimes_{G_{par}}y) \otimes_{\mathcal{A}}z$. Hence, $\psi(x \otimes_{G_{par}}(y \otimes_{\mathcal{A}}z)) = (x \otimes_{G_{par}}y) \otimes_{\mathcal{A}}z$. Note that $\psi$ and $\eta$ are mutually inverses. Finally, it is easy to see that $\eta$ determines a natural isomorphism.
\end{proof}

\subsection{Globalization of partial group actions on modules}

Analogously to set-theoretical partial group actions (see \cite{E6} for more details), we can obtain partial group actions on modules by restricting global group actions on a $K$-module $W$ to a submodule $M$ of $W$. Indeed, let $\theta: G \curvearrowright W$ be a global group action, and let $M$ be a $K$-submodule of $W$.
Define $M_g := \theta_g(M) \cap M$ and $\alpha_g:= \theta_g|_{M_{g^{-1}}}$. Then, $\alpha:= (G, M, \{ M_g \}_{g \in G}, \{ \alpha_g \}_{g \in G})$ is a partial group action of $G$ on $M$. We say that $\alpha$ is the \textbf{restriction} of $\theta$ to $M$. 

\begin{definition} \label{d: universal global partial action}
   Let $\alpha: G \curvearrowright M$ be a partial action on a $K$-module $M$, the \textbf{universal global action} of $\alpha$ is a pair $(\theta, \iota)$ where $\theta: G \curvearrowright W$ is a global group action and $\iota: M \to W$ is a $G$-equivariant map, such that for any global group action $\beta: G \curvearrowright X$ and a $G$-equivariant map $\psi: M \to X$ there exists a unique morphism $\tilde{\psi}: W \to X$ of $G$-modules such that the following diagram commutes:
   \[\begin{tikzcd}
       M & X \\
       W
       \arrow["\iota"', from=1-1, to=2-1]
       \arrow["\psi", from=1-1, to=1-2]
       \arrow["{\tilde{\psi}}"', dashed, from=2-1, to=1-2]
   \end{tikzcd}\]
\end{definition}

\begin{remark}
    Note that by a classical argument of universal property, the universal global action (if there exists) is unique up to unique isomorphism.
\end{remark}

\begin{definition}
    Let $\alpha: G \curvearrowright M$ be a partial group action and $\theta: G \curvearrowright W$ a global group action. If there exists an injective $G$-equivariant map $\iota: M \to W$ such that $\alpha$ is isomorphic to the restriction of $\theta$ to $\iota(M)$ via $\iota$ and $W = \sum_{g \in G} \theta_g(\iota(M))$, we say that $\theta$ is a \textbf{globalization} of $\alpha$.
\end{definition}

Contrary to the set-theoretical case, a globalization of partial group action on a module (if it exists) is not necessarily unique.

\begin{example} \label{e: globalization is not unique}
    Let $G := \langle g : g^{3} = 1 \rangle$, and let $\alpha:= (G, \mathbb{Z}, \{ X_h \}_{h \in G}, \{ \alpha_h \}_{h \in G})$ be the partial group action such that $X_1 = \mathbb{Z}$, $X_g = X_{g^{2}} = \{ 0 \}$, and $\alpha_1 = 1_{\mathbb{Z}}$, $\alpha_{g}(0) = \alpha_{g^2}(0)=0$. 
    Now we will construct two non-isomorphic globalizations of \(\alpha\).
    $V := \mathbb{Z} \times \mathbb{Z} \times \mathbb{Z}$, define the global action of $G$ on $V$ such that $g \cdot (x, y, z) := (z, x, y)$, thus the function $\iota: \mathbb{Z} \to V$ given by $\iota(n):= (n, 0, 0)$ is clearly a $G$-equivariant map. Moreover, note that $V$ together with $\iota$ is a globalization of $\alpha$. 

    For the second globalization consider the subgroup $N:= \{ (2n, 2n, 2n) : n \in \mathbb{Z} \}$, note that $g \cdot (2n, 2n, 2n) = (2n, 2n, 2n)$, then $W := V / N$, is a $G$-module such that $g \cdot \overline{(x, y, z)} = \overline{(z, x, y)}$, let $\phi: V \to W$ be the canonical projection. Then, $(W, \psi := \phi \circ \iota)$ is a globalization of $M$. Indeed, first observe that $\psi$ is injective. Suppose that $\psi(n) = \overline{(0, 0, 0)}$ for some $n \in \mathbb{Z}$, then $(n, 0, 0) \in N$, thus $n = 0$. Now, we have to verify that $g \cdot \overline{(n, 0, 0)} \in \psi(\mathbb{Z})$ if and only if $n = 0$, and $g^{2} \cdot \overline{(n, 0, 0)} \in \psi(\mathbb{Z})$ if and only if, $n = 0$. Note that 
    \begin{align*}
        g \cdot \overline{(n, 0, 0)} \in \psi(\mathbb{Z})
        & \Leftrightarrow \text{ there exists } m \in \mathbb{Z} \text{ such that } \overline{(0, n, 0)} = \overline{(m, 0, 0)} \\ 
        & \Leftrightarrow (-m, n, 0) \in N \\
        & \Leftrightarrow m = n = 0,
    \end{align*}
    analogously we have that $g^{2} \cdot \overline{(n, 0, 0)} \in \psi(\mathbb{Z})$ if and only if $n = 0$. Hence, the restriction of the global action on $W$ to $\psi(\mathbb{Z})$ is isomorphic to $\alpha$. Finally, observe that $V$ and $W$ are not isomorphic as $G$-modules since $W$ has torsion elements, but $V$ does not.
\end{example}

\begin{remark} \label{r: iota universal map}
    Let $\alpha: G \curvearrowright M$ be a partial action on the module $M$. Then, by Proposition~\ref{p: partial tensor product A-R bimodule}, $KG \otimes_{G_{par}}M$ is a left $G$-module. Let $\Theta: G \curvearrowright KG \otimes_{G_{par}} M$ be the associated global action of $G$ on $KG \otimes_{G_{par}}M$. Define
    \begin{align*}
        \iota: M &\to KG \otimes_{G_{par}}M \\
                m &\mapsto 1 \otimes_{G_{par}}m.
    \end{align*}
    Then, $\iota$ is a well-defined $G$-equivariant map. Indeed, observe that
    \[
        \iota(\alpha_g(m))= 1 \otimes_{G_{par}} \alpha_g(m) = g \otimes_{G_{par}}m = \Theta_{g}(\iota(m)),
    \]
    for all $g \in G$ and $m \in M_{g^{-1}}$.
\end{remark}

\begin{theorem} \label{t: existence of universal global action}
    Let $\alpha: G \curvearrowright M$ be a partial group action on the $K$-module $M$, and let $\Theta: G \curvearrowright KG \otimes_{G_{par}}M$ and $\iota: M \to KG \otimes_{G_{par}}M$ be as in Remark~\ref{r: iota universal map}. Then, $(\Theta, \iota)$ is the universal global action of $\alpha$.
\end{theorem}

\begin{proof}
        Let $X$ be a left $G$-module and $\psi: M \to X$ a $G$-equivariant map, then $\psi_{*}: KG \otimes_{G_{par}} M \to KG \otimes_{G_{par}} X$ is a map of $G$-modules.
        Note that, by Remark~\ref{r: partial tensor prouct of G-modules}, $KG \otimes_{G_{par}} X \cong KG \otimes_{KG} X \cong X$.
        Hence, we obtain a morphism of $G$-modules $\tilde{\psi}: KG \otimes_{G_{par}} M \to X$ such that $\tilde{\psi}(g \otimes_{G_{par}} m) := g \cdot \psi(m)$. It is clear that $\psi = \tilde{\psi} \circ \iota$. Finally, the uniqueness of $\tilde{\psi}$ is due to the universal property that satisfies $KG \otimes_{G_{par}} M$ (see Definition~\ref{d: universal global partial action}).
\end{proof}

\begin{remark}
    Theorem~\ref{t: existence of universal global action} and the functor \eqref{eq: partial tensor product functor i} guarantee the existence of a functor 
    \begin{equation}
        KG \otimes_{G_{par}} - : G_{par}\textbf{-Mod} \to G\textbf{-Mod}.
    \end{equation}
    Furthermore, by universal property in Definition~\ref{d: universal global partial action} the functor $KG \otimes_{G_{par}}-$ is a left adjoint to the embedding functor $G \textbf{-Mod} \to G_{par}\textbf{-Mod}$ defined in Example~\ref{e: G-modules are Gpar-modules}, i.e., for any $G_{par}$-module $M$ and $G$-module $X$ there exists a natural isomorphism
    \[
        \hom_{G}(KG \otimes_{G_{par}} M, X) \cong \hom_{G_{par}}(M, X).
    \]
\end{remark}

\textbf{Notation.} Let $\alpha: G \curvearrowright M$ be a partial group action on the module $M$. For the sake of simplicity, if there is no ambiguity, a basic element $g \otimes_{G_{par}} m$ of $KG \otimes_{G_{par}} M$ will be denoted by $\lfloor g, m \rfloor$.
Thus, for any $h \in G$, we have that $h \cdot \lfloor g, x \rfloor=\lfloor hg, x \rfloor$.

Let $\alpha: G \curvearrowright X$ be a partial group action on the $K$-module $M$, and let $\theta$ be the natural global action of $G$ on $KG$, we will denote the module $\mathcal{K}_{\theta, \alpha}$ just by $\mathcal{K}_{\alpha}$. From the definition of $\mathcal{K}_{\alpha}$ we know that
\begin{align} \label{eq: kernel globalization partial tensor product}
   \mathcal{K}_{\alpha} 
   &:= \Big\langle g \otimes x - gh^{-1} \otimes \alpha_{h}(x) : g,h \in G, \, x \in X_{h^{-1}} \Big\rangle \\
   &= \Big\langle g \otimes x - h \otimes y : x \in X_{g^{-1}h} \text{ and }\alpha_{h^{-1}g}(x) = y \Big\rangle. \notag
\end{align}

Note that in Definition~\ref{d: universal global partial action}, we do not require $\iota$ to be an injective map. In fact, $\iota$ may not be injective.

\begin{example}
    Let \(\mathcal{C}_2:= \{ 1,t : t^{2} \}\)
    Let $M$ be the free $K$-module with basis consisting of the set symbols $\{ x, y, z, u, v, w \}$, and let $G := \mathcal{C}^{3}$.
    For simplicity, we set $e=(1, 1, 1)$, $a:= (t, 1, 1)$, $b := (1, t, 1)$ and $c := (1, 1, t)$.
    Define:
    \begin{itemize}
        \item $M_e := M$ and $\alpha_e := \operatorname{id}_{M}$;
        \item $M_a := \langle x + y \rangle$, and $\alpha_{a}:= \operatorname{id}_{M_a}$;
        \item $M_{b} := \langle z - u \rangle$, and $\alpha_{b}:= \operatorname{id}_{M_b}$;
        \item $M_{c} := \langle v + w \rangle$, and $\alpha_{c}:= \operatorname{id}_{M_{c}}$;
        \item $M_{abc} := \{ 0 \}$, and $\alpha_{abc}:= \operatorname{id}_{M_{abc}}$;
        \item $M_{ab} := \langle x, u \rangle$, and $\alpha_{ab}: M_{ab} \to M_{ab}$, such that $\alpha_{ab}(x)=u$ and $\alpha_{ab}(u)=x$;
        \item $M_{ac} := \langle v, y \rangle$, and $\alpha_{ac}: M_{ac} \to M_{ac}$, such that $\alpha_{ac}(y)=v$ and $\alpha_{ac}(v)=y$;
        \item $M_{bc} := \langle z, w \rangle$, and $\alpha_{bc}: M_{bc} \to M_{bc}$, such that $\alpha_{bc}(z)=w$ and $\alpha_{bc}(w)=z$.
    \end{itemize}
    It is clear that $\alpha_{g}$ is a $K$-linear isomorphism for all $g \in G$. To verify that $\alpha$ is a partial group action observe that    
    \begin{equation} \label{eq: non-globalizable partial group action formula}
        M_{g} \cap M_{h} =
        \left\{\begin{matrix}
         M_h & \text{ if } g = e \\ 
         M_g = M_h & \text{ if } g = h \\ 
         M_g & \text{ if } h = e \\ 
          \{ 0 \}  & \text{ otherwise.}
        \end{matrix}\right.
    \end{equation}
    Then,     
    \[
        \alpha_g(M_{g^{-1}} \cap M_{g^{-1}h}) =
        \left\{\begin{matrix}
            M_h & \text{ if } g = e \\ 
            M_g = M_h & \text{ if } g = h \\ 
            M_g & \text{ if } h = e \\ 
            \{ 0 \}  & \text{ otherwise,}
        \end{matrix}\right.
    \]
    what verifies $(ii)$ of Definition~\ref{d: partial group action}. Observe that $(iii)$ of Definition~\ref{d: partial group action} is satisfied for $m = 0$. Otherwise, if there exists $m \in M_{g^{-1}} \cap M_{g^{-1}h^{-1}}$ such that $m \neq 0$, then by equation \eqref{eq: non-globalizable partial group action formula} we know that $g^{-1}=e$ or $h^{-1}=e$ or $g^{-1}h^{-1} = e$, then 
    \begin{itemize}
        \item if $g^{-1}=1$, then $m \in M_{h^{-1}}$ and $\alpha_h \alpha_1(m) = \alpha_h(m)$,
        \item if $h^{-1}=1$, then $m \in M_{g^{-1}}$ and $\alpha_1 \alpha_g(m) = \alpha_g(m)$,
        \item if $g^{-1}h^{-1}=1$, then $m \in M_{g^{-1}}$, $g^{-1} = h$ and $\alpha_{g^{-1}} \alpha_g(m) = m = \alpha_1(m)$.
    \end{itemize}
    Therefore, $\alpha$ is a well-defined partial group action of $G$ on $M$. Now define
    \[
       q := x + y + z - u - v - w \in M.
    \]
    It is clear that $q \neq 0$. Note that we can rewrite $q$ as follows
    \begin{align*}
        q
        &= (x + y) + (z - u) - (v + w) \\
        &= \alpha_a(x + y) + \alpha_b(z - u) - \alpha_c(v + w) \\
        &= \alpha_a(x + y) + \alpha_b(z - \alpha_{ab}(x)) - \alpha_c(\alpha_{ac}(y) + \alpha_{bc}(z)).
    \end{align*}
    Then,
    \begin{align*}
        \lfloor 1,q \rfloor
        &= \lfloor 1, \alpha_a(x + y) \rfloor + \lfloor 1, \alpha_b(z - \alpha_{ab}(x)) \rfloor - \lfloor 1, \alpha_c(\alpha_{ac}(y) + \alpha_{bc}(z)) \rfloor \\
        &= \lfloor a, x + y \rfloor + \lfloor b, z - \alpha_{ab}(x) \rfloor - \lfloor c, \alpha_{ac}(y) + \alpha_{bc}(z) \rfloor \\
        &= \lfloor a, x \rfloor + \lfloor a, y \rfloor + \lfloor b, z \rfloor - \lfloor b, \alpha_{ab}(x) \rfloor - \lfloor c, \alpha_{ac}(y) \rfloor - \lfloor c, \alpha_{bc}(z) \rfloor \\
        &= \lfloor a, x \rfloor + \lfloor a, y \rfloor + \lfloor b, z \rfloor - \lfloor bab, x \rfloor - \lfloor cac, y \rfloor - \lfloor cbc, z \rfloor \\
        &= \lfloor a, x \rfloor + \lfloor a, y \rfloor + \lfloor b, z \rfloor - \lfloor a, x \rfloor - \lfloor a, y \rfloor - \lfloor b, z \rfloor \\
        &= 0.
    \end{align*}
    Hence, the map $\iota: M \to KG \otimes_{G_{par}} M$ is not injective.
\end{example}

\begin{proposition}
    Let $\alpha: G \curvearrowright M$ be a partial group action on the $K$-module $M$. Suppose that there exists a global action $\theta: G \curvearrowright W$ and an injective $G$-equivariant map $\psi: M \to W$ such that $\alpha$ is the restriction of $\theta$ to $M$. Then, the map $\iota$ is injective and the canonical global action of $G$ on $KG \otimes_{G_{par}} M$ is a globalization of $\alpha$, we call this globalization as the \textbf{universal globalization}
\end{proposition}

\begin{proof}
    By the universal property of $KG \otimes_{G_{par}}M$ the following diagram is commutative
    \[\begin{tikzcd}
        M & W \\
        {KG \otimes_{G_{par}}M}
        \arrow["\psi", hook, from=1-1, to=1-2]
        \arrow["\iota"', from=1-1, to=2-1]
        \arrow["{\hat{\psi}}"', curve={height=12pt}, dashed, from=2-1, to=1-2]
    \end{tikzcd}\]
    Note that the injectivity of $\psi$ implies the injectivity of $\iota$. Let $\beta:= (G, \iota(M), \{ D_g \}_{g \in G}, \{ \beta_g \}_{g \in G})$ be the restriction of the global action $\Theta: G \curvearrowright KG \otimes_{G_{par}} M$ to $\iota(M)$, since $\iota$ is a $G$-equivaliant map, it is clear that $\iota|_{\iota(M)}: M \to \iota(M)$ is a bijective $G$-equivariant map, thus we only have to verify that $D_g \subseteq \iota(M_g)$ for all $g \in G$. Let $g \in G$  and $m \in M$ such that $\lfloor g, m \rfloor  \in \iota(M)$, then $\theta_{g}(\psi(m)) = \hat{\psi}(\lfloor g, m \rfloor ) \in \psi(M) \cap \theta_{g}(\psi(M))$, whence we conclude that $m \in M_{g^{-1}}$ since $\alpha$ is the restriction of $\theta$, thus $\lfloor g, m \rfloor  = \lfloor 1, \theta_g(m) \rfloor  \in \iota(M_g)$. Therefore, $KG \otimes_{G_{par}} M$ is a globalization of $\alpha: G \curvearrowright M$.
\end{proof}

\begin{corollary}
    Let $\alpha: G \curvearrowright M$ be a partial group action. If the map $\iota: M \to KG \otimes_{G_{par}}M$ is not injective, then $\alpha$ is not globalizable.
\end{corollary}

\begin{proposition} \label{p: universal global action is a globalization iff}
    Let $\alpha: G \curvearrowright M$ be a partial group action. Then, the universal global action $\Theta: G \curvearrowright KG \otimes_{G_{par}}M$ is a globalization if and only if for all $x, y \in M$ and $g \in G$ such that $ \lfloor g, x \rfloor  =  \lfloor 1, y \rfloor $, we have that $x \in M_{g^{-1}}$ and $\alpha_g(x)=y$.
\end{proposition}
\begin{proof}
    Recall that $\iota: M \to KG \otimes_{G_{par}}M$ is the canonical morphism defined in Remark~\ref{r: iota universal map}.
    Let $\beta:= (G, \iota(M), \{ D_g \}_{g \in G}, \{ \alpha_{g} \}_{g \in G})$ be the restriction of $\Theta$ to $\iota(M)$.
    Set $\hat{\iota}:= \iota|_{\iota(M)}: M \to \iota(M)$.
    Suppose that $\Theta$ is a globalization of $\alpha$, then $\hat{\iota}$ is an isomorphism of partial group actions.
    Let $x, y \in M$ and $g \in G$ such that $ \lfloor g, x \rfloor  =  \lfloor 1, y \rfloor  \in \iota(M)$, since $\beta$ is the restriction of $\Theta$ and $\hat{\iota}$ is an isomorphism we conclude that $x \in M_{g^{-1}}$.
    Then, $ \lfloor 1, y \rfloor  =  \lfloor g, x \rfloor  =  \lfloor 1, \alpha_g(x) \rfloor $, by the injectivity of $\iota$ we obtain that $y = \alpha_g(x)$.
    For the converse we want to verify that $\hat{\iota}$ is an isomorphism of partial actions.
    Let $x \in M$ such that $\iota(x)=0$, then $ \lfloor 1,x \rfloor  = 0 =  \lfloor 1,0 \rfloor $, thus $x = \alpha_1(0)=0$.
    Therefore, the map $\iota$ is injective.
    Hence, to verify that $\hat{\iota}$ is an isomorphism of partial actions we only have to show that $\iota(M_g) = D_g$.
    Let $ \lfloor g, x \rfloor  \in D_g \subseteq \iota(M)$, since $ \lfloor g, x \rfloor  \in \iota(M)$ there exists $z \in M$ such that $ \lfloor g, x \rfloor  =  \lfloor 1, z \rfloor $, thus by the hypotheses we get that $x \in M_{g^{-1}}$, and therefore $ \lfloor g, x \rfloor  =  \lfloor 1, \alpha_g(x) \rfloor  = \iota(\alpha_g(x))$.
    Hence, $\iota(M_g)= D_g$.
\end{proof}

\begin{remark} \label{r: tau map for partial representations}
    For any $K_{par}G$-module we can define a $K$-linear map $\tau_0: KG \otimes X \to X$ such that $\tau_0(g \otimes x) := [g] \cdot x$, note that $g \otimes [h] \cdot x - gh \otimes e_{h^{-1}} \cdot x \in \ker \tau_0$ for all $x \in X$ and $g, h \in G$. Then, by the universal property of the partial tensor product (Proposition~\ref{p: universal property of partial tensor product}) there exists a $K$-linear map $\tau: KG \otimes_{G_{par}} X \to X$ such that $\tau(\lfloor g, x \rfloor) = [g] \cdot x$. Furthermore, note that $\tau \circ \iota = 1_{X}$.
\end{remark}

\begin{theorem} \label{t: globalization of partial representations}
    Let $\pi: G \to \operatorname{End}_{K}(M)$ be a partial group representation, and let $\alpha: G \curvearrowright M$ be the partial group action induced by $\pi$. Then, $\Theta: G \curvearrowright KG \otimes_{G_{par}}M$ the universal globalization of $\alpha$.
\end{theorem}
\begin{proof}
    Let $g \in G$, $x, y \in M$ such that $[g, x] = [1, z]$. By Remark~\ref{r: tau map for partial representations}, $[g] \cdot x = \tau([g, x]) = \tau([1,z])= z$.
    Therefore, $[g, x] = [1, [g] \cdot x] = [g, e_{g^{-1}} \cdot x]$. Thus, $[1, x] = [1, e_{g^{-1}} \cdot x]$. Applying $\tau$ again we obtain $x = e_{g^{-1}} \cdot x$. Then, $x \in M_{g^{-1}}$ and $\alpha_g(x)= [g] \cdot x = z$. Hence, by Proposition~\ref{p: universal global action is a globalization iff}, $\Theta$ is a globalization of $\alpha: G \curvearrowright M$.
\end{proof}

Note that the fact that a partial group action is determined by a partial representation does not imply uniqueness of globalization. In fact, the partial action defined in Example~\ref{e: globalization is not unique} arises from a partial representation.

\begin{remark}[\textbf{Dilation of partial representations}] \label{r: dilation of partial representations}
    Let \(\pi: G \to \operatorname{End}_{K}(M)\) be a partial group representation.  
    Observe that if we define \(T: KG \otimes_{G_{par}} M \to KG \otimes_{G_{par}} M\) by \(T := \iota \circ \tau\), then \((KG \otimes_{G_{par}} M, T, \iota)\) is a dilation of \((M, \pi)\) in the sense of \cite[Definition 4.1]{ABV2}.  
    Furthermore, by Theorem~\ref{t: globalization of partial representations}, \(KG \otimes_{G_{par}} M\) is a universal globalization. Hence, by \cite[Proposition 4.7]{ABV2}, we conclude that \(KG \otimes_{G_{par}} M\) naturally isomorphic to the standard dilation of \(M\).  
\end{remark}

\begin{example}
    We will see in Lemma~\ref{l: KG o KparG left isormorphism} and Remark~\ref{r: K otimes KparG is B} that the universal globalization of the partial group algebra $K_{par}G$ is isomorphic as $KG$-$K_{par}G$-bimodule to $KG \otimes \mathcal{B}$, where $KG \otimes \mathcal{B}$ is a $KG$-$K_{par}G$-module such that
    \[
        g \cdot (h \otimes w) = gh \otimes w \text{ and } (h \otimes w )\cdot [g] = hg \otimes (w \triangleleft [g]).
    \]
    Furthermore, $KG \otimes \mathcal{B}$ has a structure of algebra, with the product given by
    \[
        (g \otimes w) * (h \otimes u) := gh \otimes (w \triangleleft [h]) u.
    \]
    One can verify that $ \lfloor 1, 1 \rfloor  * KG \otimes \mathcal{B} \cong \mathcal{B} \rtimes G \cong K_{par}G$.
\end{example}

\section{A globalization problem for partial group (co)homology}

Recall that in Example~\ref{e: partial action induced by a partial representation} we obtained a functor
\[
   \Upsilon: K_{par}G \textbf{-Mod} \to G_{par}\textbf{-Mod}.
\]
Thus, by Theorem~\ref{t: globalization of partial representations}, when we compose $\Upsilon$ with the functor
\[
   KG \otimes_{G_{par}} - : G_{par}\textbf{-Mod} \to G \textbf{-Mod},
\]
(as outlined in Remark~\ref{r: partial tensor product for partial representations}), we obtain a functor
\[
    \Lambda:= (KG \otimes_{G_{par}} - ) \circ \Upsilon : K_{par}G \textbf{-Mod} \to G \textbf{-Mod}.
\]
that maps a $K_{par}G$-module $M$ to the universal globalization $\Theta: G \curvearrowright KG \otimes_{G_{par}} M$ of the $G_{par}$-module $M$. We refer to this functor as the \textbf{globalization functor}. \index{Globalization functor}

\begin{lemma} \label{l: zeros of K charaterization}
   Let $X$ be a left $K_{par}G$-module, suppose that $\{ z_i \}_{i = 0}^{n} \subseteq X$ and $\{ g_i \}_{i = 1}^{n} \subseteq G$ are such that
   \[
      1 \otimes_{G_{par}} z_0 + \sum_{i = 1}^{n} g_i \otimes_{G_{par}} z_i = 0.
   \]
   Then, $z_0 = - \sum_{i=1}^{m} [g_i] \cdot z_i$.
\end{lemma}
\begin{proof}
   Note that, by Remark~\ref{r: tau map for partial representations}, we have
   \[
       0 = \tau \Big(1 \otimes_{G_{par}} z_0 + \sum_{i = 1}^{n} g_i \otimes_{G_{par}} z_i \Big) = z_0 + \sum_{i = 1}^{n} [g_i] \cdot z_i,
   \]
   whence we obtain the desired conclusion.
\end{proof}

\begin{theorem} \label{t: the globalization is exact}
    The globalization functor $\Lambda : K_{par}G\textbf{-Mod} \to G\textbf{-Mod}$ is exact.
\end{theorem}

\begin{proof}
    By Proposition~\ref{p: partial tensor product is right exact} we know that $\Lambda$ is right exact, thus we only have to prove that $\Lambda$ preserve injective morphisms. 
    Let $X$ and $Y$ be left $K_{par}G$-modules, $\psi: X \to Y$ an injective map of $K_{par}G$-modules, and let $\psi_*: \Lambda(X) \to \Lambda(Y)$ be the morphism of $G$-modules determined by $\psi$.
    It is worth to note that $\psi_*$ is given by $\psi_*( \lfloor g, x \rfloor )=  \lfloor g, \psi(x) \rfloor $.
    Recall that for any $z \in KG \otimes_{G_{par}} X$ there exists $\{ g_i \}_{i=0}^{n} \subseteq G$ and $\{ x_i \}_{i=0}^{n} \subseteq X$ such that $z = \sum_{i = 0}^{n} \lfloor g_i, x_i \rfloor $. Then, the injectivity of $\psi_*$ is equivalent to:
    \begin{enumerate}[($\flat$)]
        \item For all $n \in \mathbb{N}$, $\{ g_i \}_{i=0}^{n} \subseteq G$ and $\{ x_i \}_{i=0}^{n} \subseteq X$ such that $\sum_{i = 0}^{n}  \lfloor g_i, \psi(x_i) \rfloor  = 0$ we have that $\sum_{i = 0}^{n}  \lfloor g_i, x_i \rfloor  = 0$.
    \end{enumerate}
    We proceed  to prove $(\flat)$ with an induction argument over $n \in \mathbb{N}$. For $n = 0$, note that if there exists $x \in X$ and $g \in G$ such that $ \lfloor g, \psi(x) \rfloor  = 0$, then $ \lfloor 1, \psi(x) \rfloor  = 0$. Since $\iota: Y \to KG \otimes_{G_{par}} Y$ is injective, we have $\psi(x) = 0$, consequently $x = 0$ and $ \lfloor g, x \rfloor  = 0$, this give us the base of the induction. Let $n \in \mathbb{N}$, suppose that for all  $\{ h_i \}_{i = 0}^{n} \subseteq G$ and $\{ y_i \}_{i=0}^{n} \subseteq X$ such that $\psi_* \left(\sum_{i = 0}^{n}  \lfloor h_i, y_i \rfloor  \right) = \sum_{i = 0}^{n}  \lfloor h_i, \psi(y_i) \rfloor  = 0$ we have that $\sum_{i = 0}^{n}  \lfloor h_i, y_i \rfloor  = 0$. Let $z = \sum_{i = 0}^{n+1}  \lfloor g_i, x_i \rfloor $ be such that $\psi_*(z)=0$, then
   \[
       0 = g_{n+1}^{-1} \cdot \psi_*(z) =  \lfloor 1, \psi(x_{n+1}) \rfloor  + \sum_{i = 0}^{n}  \lfloor g_{n+1}^{-1}g_i, \psi(x_i) \rfloor .
   \]
   Then, by Lemma~\ref{l: zeros of K charaterization}  we have that
   \[
       \psi(x_{n+1}) = -\sum_{i=0}^{n} [g_{n+1}^{-1} g_i] \cdot \psi(x_i) = \psi\left( -\sum_{i=0}^{n} [g_{n+1}^{-1} g_i] \cdot x_i \right).
   \]
    Since $\psi$ is injective we conclude that $x_{n+1} = - \sum_{i = 0}^{n} [g_{n+1}^{-1}g_i] \cdot x_i$. Therefore, 
    \begin{align*}
         \lfloor 1, x_{n+1} \rfloor  
        &= - \sum_{i = 0}^{n} \, \lfloor 1, \,  [g_{n+1}^{-1}g_i ] \cdot x_i  \rfloor  \\
        &= - \sum_{i = 0}^{n} \, \lfloor  g_{n+1}^{-1}g_i, \, e_{g_i^{-1}g_{n+1}} \cdot x_i \rfloor .
    \end{align*}
    Then,
   \[
       g_{n+1}^{-1} \cdot z 
       =  \lfloor 1, x_{n+1} \rfloor  + \sum_{i = 0}^{n} \, \big \lfloor g_{n+1}^{-1}g_i, \, x_i \big \rfloor  
       = \sum_{i = 0}^{n}\Big \lfloor g_{n+1}^{-1}g_i, \, x_i - e_{g_i^{-1}g_{n+1}} \cdot x_i \Big \rfloor .
   \]
    Observe that $\psi_*(g_{n+1}^{-1} \cdot z) = g_{n+1}^{-1} \cdot \psi_*(z) = 0$, then by the induction hypotheses we have that $g_{n+1}^{-1} \cdot z = 0$, which implies $z = 0$. Therefore, $\psi_*$ is injective.
\end{proof}

Recall that by definition \(\Lambda(K_{par}G)\) is the \(KG\)-module \(KG \otimes_{G_{par}} K_{par}G\), where \(K_{par}G\) possesses its natural left \(G_{par}\)-module structure given by \(\Upsilon\).

\begin{corollary} \label{c: the globalizer module is KparG flat}
    $KG \otimes_{G_{par}} K_{par}G$ is flat as right $K_{par}G$-module.
\end{corollary}
\begin{proof}
    Let $0 \to A \to B \to C \to 0$ be an exact sequence of left $K_{par}G$-modules.
    By Proposition~\ref{p: partial tensor product is associative}, the complex 
    \[
        0 \to (KG \otimes_{G_{par}} K_{par}G) \otimes_{K_{par}G}A \to (KG \otimes_{G_{par}}  K_{par}G) \otimes_{K_{par}G} B \to (KG \otimes_{G_{par}}  K_{par}G) \otimes_{K_{par}G} C \to 0
    \]
    is isomorphic to
    \[
        0 \to KG \otimes_{G_{par}} A \to KG \otimes_{G_{par}} B \to KG \otimes_{G_{par}} C \to 0,
    \]
    and, by Theorem~\ref{t: the globalization is exact}, this complex is exact.
\end{proof}

\begin{remark}[\textbf{Exactness of dilation functor}]
    By Remark~\ref{r: dilation of partial representations}, the dilation functor in \cite[Proposition 4.5]{ABV2} coincides with the standard dilation functor. Therefore, by Theorem~\ref{t: the globalization is exact}, the standard dilation functor is exact, making \cite[Proposition 5.2]{ABV2} applicable to group algebras (see \cite[Remark 5.3]{ABV2}).
    Note that Theorem~\ref{t: the globalization is exact} can also be derived from \cite[Proposition 5.2]{ABV2} and Proposition~\ref{p: partial tensor product is right exact}. However, for the sake of completeness, we have provided an alternative proof.
\end{remark}

\subsection{Partial group homology}

In this section we prove that the partial group homology of a $K_{par}G$-module $M$ coincides with the group homology of the universal globalization $\Theta: G \curvearrowright KG \otimes_{G_{par}}M$.

\begin{lemma} \label{l: KG o KparG left isormorphism}
    The modules $KG \otimes_{G_{par}} K_{par}G$ and $KG \otimes \mathcal{B}$ are isomorphic as left $KG$-modules. In particular, $KG \otimes_{G_{par}} K_{par}G$ is free as left $KG$-module. 
\end{lemma}
\begin{proof}
    Recall that $K_{par}G$ is the $K$-algebra generated by the Exel's semigroup $\mathcal{S}(G)$ of $G$ and by Proposition~\ref{p: basic decomposition of elements of SG} any element $z \in \mathcal{S}(G)$ has a unique representation $z = [g]u$, where $g \in G$ and $u$ is an idempotent element of $\mathcal{S}(G)$.
    Thus, we can define a $K$-linear map $f_0: KG \otimes K_{par}G \to KG \otimes \mathcal{B}$ such that $f_0(g \otimes [h]u) = gh \otimes e_{h^{-1}}u$.
    Moreover, note that by the definition of $f_0$ we have that $f_0(gh \otimes e_{h^{-1}}u) = f_0(g \otimes [h] u)$.
    Therefore, by Proposition~\ref{p: universal property of partial tensor product} there exists a $K$-linear map $f: KG \otimes_{G_{par}} K_{par}G \to KG \otimes \mathcal{B}$ such that $f(g \otimes_{G_{par}} [h]u)=gh \otimes e_{h^{-1}}u$, it is clear that this function is also a morphism of $G$-modules.
    Conversely, recall that the idempotent elements of $\mathcal{S}(G)$ form a $K$-basis for $\mathcal{B}$ as $K$-module.
    Thus, there exists a $K$-linear map $f': KG \otimes \mathcal{B} \to KG \otimes_{G_{par}} K_{par}G$ such that $f'(g \otimes u) = g \otimes_{G_{par}} u$. It is easy to see that $f \circ f' = 1_{KG \otimes \mathcal{B}}$, on the other hand observe that
    \[
        f' \circ f (g \otimes_{G_{par}} [h]u) = f'(gh \otimes e_{h^{-1}}u) = gh \otimes_{G_{par}} e_{h^{-1}} u = g \otimes_{G_{par}} [h]u.    
    \]
    Thus, $f'$ and $f$ are mutually inverses.
\end{proof}

\begin{remark} \label{r: K otimes KparG is B}
    Consider the isomorphism of left $KG$-modules $f: KG \otimes_{G_{par}} K_{par}G \to KG \otimes \mathcal{B}$ defined in the proof of Lemma~\ref{l: KG o KparG left isormorphism}. Note that the right $K_{par}G$-module structure on $KG \otimes \mathcal{B}$ induced by $f$ is given by:
    \[
        (g \otimes w) \cdot [h] = gh \otimes (w \triangleleft [h]),
    \]
    with that $K_{par}G$-module structure on $KG \otimes \mathcal{B}$ in mind and considering $K$ a trivial $KG$-module we obtain the following isomorphisms of right $K_{par}G$-modules
    \[
        K \otimes_{G_{par}} K_{par}G \cong K \otimes_{KG} KG \otimes_{G_{par}} K_{par}G \overset{f_*}{\cong} K \otimes_{KG} KG \otimes \mathcal{B} \cong K \otimes \mathcal{B} \cong \mathcal{B}.
    \]
    Explicitly, this is isomorphism is given by $r \otimes_{G_{par}} z \mapsto r \varepsilon(z)$ for all $r \in K$ and $z \in K_{par}G$. Let $\hat{\varepsilon}: KG \to K$ be such that $\hat{\varepsilon}(g) = 1$, then applying the functor
    \[
       - \otimes_{KG} KG \otimes_{G_{par}} K_{par}G : \textbf{Mod-}G \to K\textbf{-Mod}
    \]
    to $\hat{\varepsilon}$ we obtain a surjective morphism of right $K_{par}G$-modules
    \[
       \varepsilon': KG \otimes_{G_{par}} K_{par}G \to  \mathcal{B}
    \]
    such that $\varepsilon'([g, z]) = \varepsilon(z)$.
\end{remark}

\begin{theorem} \label{t: partial group homology is global group homology}
    Let $\pi: G \to \operatorname{End}_{K}(M)$ be a partial group representation. Then,
    \[
        H_{\bullet}^{par}(G, M) \cong H_{\bullet}(G, \Lambda(M)).
    \]
\end{theorem}

\begin{proof}
    Let $P_\bullet \to K$ be a projective resolution of $K$ in $G$-\textbf{Mod}. Since $KG \otimes_{G_{par}} K_{par}G$ is free as left $KG$-module we obtain a resolution $P_\bullet \otimes_{KG} (KG \otimes_{G_{par}} K_{par}G) \to K \otimes_{KG}(KG \otimes_{G_{par}} K_{par}G)$. By Remark~\ref{r: K otimes KparG is B} and the fact that $KG \otimes_{G_{par}} K_{par}G$ is flat as right $K_{par}G$-module, we have that $P_\bullet \otimes_{KG}(KG\otimes_{G_{par}} K_{par}G) \to \mathcal{B}$ is a flat resolution of right $K_{par}G$-modules of $\mathcal{B}$. By \cite[Lemma 3.2.8]{weibel_1994} we can compute $H^{par}_{\bullet}(G, M) = \operatorname{Tor}_{\bullet}^{K_{par}G}(\mathcal{B}, M)$ using a flat resolution of $\mathcal{B}$. Then,
    \begin{align*}
        H_{\bullet}^{par}(G, M) 
        & \cong H_\bullet(P_\bullet \otimes_{KG} (KG \otimes_{G_{par}} K_{par}G) \otimes_{K_{par}G} M) \\
        (\text{by Proposition~\ref{p: partial tensor product is associative}})& \cong H_\bullet(P_\bullet \otimes_{KG} (KG \otimes_{G_{par}} M)) \\
        & \cong H_\bullet(G, KG \otimes_{G_{par}} M) \\
        & \cong H_\bullet(G, \Lambda(M)).
    \end{align*}
\end{proof}

A direct application of Theorem~\ref{t: partial group homology is global group homology} give us the following two corollaries:

\begin{corollary}[Lyndon-Hochschild-Serre spectral sequence]
    Let $\pi: G \to \operatorname{End}_K(M)$ be a partial representation, and let $N$ be a normal subgroup of $G$. Then, there exists a spectral sequence
    \[
        H_{p}(G/N, H_q(N, \Lambda(M))) \Rightarrow H_{p+q}^{par}(G, M).
    \]
\end{corollary}

\begin{proof}
    The spectral sequence is obtained using the Lyndon-Hochschild-Serre spectral sequence for $\Lambda(M)$ and applying Theorem~\ref{t: partial group homology is global group homology}.
\end{proof}

\begin{corollary}[Shapiro's lemma]
    Let $G$ be a group, let $S$ be a subgroup of $G$, and let $\pi: S \to \operatorname{End}_K(M)$ be a partial representation. Then,     
    \[
        H_{\bullet}^{par}(S, M) \cong H_\bullet(G, KG \otimes_{S_{par}} M).
    \]
\end{corollary}
\begin{proof}
    Using the classical Shapiro's lemma and Theorem~\ref{t: partial group homology is global group homology} we obtain the following natural isomorphisms
    \begin{align*}
        H^{par}_{\bullet}(S, M) 
        &\cong H_{\bullet}(S, KS \otimes_{S_{par}} M) \\
        (\text{by Shapiro's lemma}) &\cong H_{\bullet}(G, KG \otimes_{KS}(KS \otimes_{S_{par}} M)) \\
        (\text{by Proposition~\ref{p: partial tensor product is associative}})&\cong H_{\bullet}(G, (KG \otimes_{KS}KS) \otimes_{S_{par}} M) \\
        &\cong H_{\bullet}(G, KG \otimes_{S_{par}} M).
    \end{align*}
\end{proof}

\begin{remark}
    Using Theorem~\ref{t: partial group homology is global group homology} and a classical argument of spectral sequences of bicomplexes or by utilizing the Acyclic Assembly Lemma \cite[Lemma 2.7.3]{weibel_1994} as in the proof of \cite[Theorem 2.7.2]{weibel_1994} one can prove that the inclusion map $\iota_\bullet: C^{par}_{\bullet}(G, M) \to C_{\bullet}(G, KG \otimes_{G_{par}}M)$ is a quasi-isomorphism.
    The proof of this fact is omitted since it is quite technical and does not add any new insight to the main results of this paper.
\end{remark}

\subsection{Partial group cohomology}
We now study the cohomological framework, extending the implications of Theorem~\ref{t: partial group homology is global group homology} by dualizing it into a cohomology spectral sequence that converges to the partial group cohomology.

Let $M$ be a right $K_{par}G$-module, by the tensor-hom adjunction and Remark~\ref{r: K otimes KparG is B} we obtain the following isomorphisms:
\begin{align*}
    \hom_{K_{par}G}(\mathcal{B}, M) 
    &\cong \hom_{K_{par}G}(K \otimes_{G_{par}} K_{par}G, M) \\
    &\cong \hom_{K_{par}G}\big(K \otimes_{KG} (KG \otimes_{G_{par}} K_{par}G), M \big) \\
    &\cong \hom_{KG}\big(K, \hom_{K_{par}G}(KG \otimes_{G_{par}} K_{par}G, M) \big).
\end{align*}
Recall that by Lemma~\ref{l: KG o KparG left isormorphism} we have that $KG \otimes_{G_{par}} K_{par}G$ is free as left $KG$-module. Thus, if $Q$ is an injective right $K_{par}G$-module by \cite[Lemma 3.5]{Lam1998LecturesOM} the right $KG$-module $\hom_{K_{par}G}(KG \otimes_{G_{par}}K_{par}G, Q)$ is injective. Therefore, we have that $\hom_{K_{par}G}( \mathcal{B}, - ): \textbf{Mod-}K_{par}G \to K\textbf{-Mod}$ is isomorphic to the composition of functors
\[
   \hom_{KG}(K, - ) \circ \hom_{K_{par}G}(KG \otimes_{G_{par}} K_{par}G, -),
\]
and that the functor $\hom_{K_{par}G},( KG \otimes_{G_{par}} K_{par}G, -)$ sends injective right $K_{par}G$-modules to injective right $KG$-modules. Hence, by \cite[Theorem 5.8.3]{weibel_1994} we obtain the following
\begin{proposition} \label{p: cohomological globalization spectral sequence}
    Let $M$ be a right $K_{par}G$-module. Then, there exist a cohomology spectral sequence
    \begin{equation}
        E^{p,q}_{2} = H^{p}(G, \operatorname{Ext}^{q}_{K_{par}G}(\Lambda(K_{par}G), M)) \Rightarrow H^{p+q}_{par}(G, M).
        \label{eq: cohomological spectral sequence}
    \end{equation}
\end{proposition}

To understand when the spectral sequence \eqref{eq: cohomological spectral sequence} collapses, we first need to establish under what conditions $KG \otimes_{G_{par}} K_{par}G$ is projective as right $K_{par}G$-module. With this purpose in mind, we will identify $KG \otimes_{G_{par}} K_{par}G$ with $KG \otimes \mathcal{B}$ as in Remark~\ref{r: K otimes KparG is B}. 

\begin{lemma} \label{l: computation rules for nu}
    Let $G$ be a group, for all $g \in G$ we define $\nu_g := 1 - e_g$. Then,
    \begin{enumerate}[(i)]
        \item $\nu_g \nu_h = \nu_h \nu_g$ for all $g, h \in G$, 
        \item $[g]\nu_h = \nu_{gh}[g]$ for all $g, h \in G$,
        \item $\nu_g[g]= [g] \nu_{g^{-1}} = 0$ for all $g \in G$,
        \item $K_{par}G = e_g K_{par}G \oplus \nu_g K_{par}G = K_{par}G e_g \oplus K_{par}G \nu_g$ for all $g \in G$.
    \end{enumerate}
\end{lemma}
\begin{proof}
    Items $(i)$, $(ii)$, and $(iii)$ are direct consequences of Proposition~\ref{p: computations rules}. Finally, $(iv)$ follows from $(iii)$ and the obvious fact $1 = e_g + \nu_g$.
\end{proof}

\begin{lemma} \label{l: decomposition of KG o B}
    Let $\mathcal{N}$ be the right $K_{par}G$-submodule of $KG \otimes \mathcal{B}$ generated by the set $\{ g \otimes \nu_{g^{-1}} : g \in G \}$. Then, $KG \otimes \mathcal{B} \cong K_{par}G \oplus \mathcal{N}$ as right $K_{par}G$-modules.
\end{lemma}
\begin{proof}
    Consider the maps of right $K_{par}G$-modules $\psi_0 : KG \otimes \mathcal{B} \to K_{par}G$ such that $\psi_{0}(g \otimes u) = [g]u$ and $\phi_0: K_{par}G \to KG \otimes \mathcal{B}$ such that $\phi_0([g]u) = (1 \otimes 1) \triangleleft [g]u = g \otimes e_{g^{-1}}u$. Observe that
    \[
        \psi_0 \phi_0([g]u) = \psi_0(g \otimes e_{g^{-1}}u) = [g]e_{g^{-1}}u = [g]u \text{ for all } g \in G \text{ and } u \in E(\mathcal{S}(G)).
    \]
    Then, $\psi_0 \phi_0 = 1_{K_{par}G}$. Note that $\im \phi_0$ is generated as right $K_{par}G$-submodule by the set $\{ \phi_0([g]) = g \otimes e_{g^{-1}} : g \in G \}$. By Lemma~\ref{l: computation rules for nu} we know that $g \otimes u = g \otimes e_{g^{-1}}u + g \otimes \nu_{g^{-1}}u$ for all $g \in G$ and $u \in \mathcal{B}$, then  $KG \otimes \mathcal{B} = \im \phi_0 + \mathcal{N}$.  On the other hand, it is clear that $\mathcal{N} \subseteq \ker \psi_0$. Therefore, $\im \phi_0 \cap \mathcal{N} = 0$, thus $KG \otimes \mathcal{B} = \mathcal{N} \oplus \im \phi_0$. Furthermore, since $\phi_0$ is injective we have that $K_{par}G \overset{\phi_0}{\cong} \im \phi_0$ as right $K_{par}G$-modules.   
\end{proof}

\begin{lemma} \label{l: N is projective iff}
    Consider $KG \otimes K_{par}G$ as the free right $K_{par}G$-module with the usual right action inherited from $K_{par}G$. Define $\delta: KG \otimes K_{par}G \to \mathcal{N}$ as the morphism of right $K_{par}G$-modules such that $\delta(g \otimes z):= (g \otimes \nu_{g^{-1}}) \triangleleft [z]$.
    Then, the following statements are equivalent:
    \begin{enumerate}[(i)]
        \item $KG \otimes \mathcal{B}$ is projective as right $K_{par}G$-module,
        \item $\mathcal{N}$ is projective as right $K_{par}G$-module,
        \item there exists a morphism of right $K_{par}G$-modules $\phi: \mathcal{N} \to KG \otimes K_{par}G$ such that $\delta \phi = 1_\mathcal{N}$.
    \end{enumerate}
\end{lemma}
\begin{proof}
    By Lemma~\ref{l: decomposition of KG o B} we have that $KG \otimes \mathcal{B} \cong K_{par}G \oplus \mathcal{N}$, then $KG \otimes \mathcal{B}$ is projective if, and only if $\mathcal{N}$ is projective. 
    For the item $(iii)$, consider the following exact sequence of right $K_{par}G$-modules
    \[
        0 \to \ker \delta \to KG \otimes K_{par}G \overset{\delta}{\to} \mathcal{N} \to 0.
    \]
    Then, it is clear that $(ii)$ implies $(iii)$. On the other hand, the existence of the map $\phi$ implies that the exact sequence splits and thus $\mathcal{N}$ is a direct summand of the right free $K_{par}G$-module $KG \otimes K_{par}G$.
\end{proof}

\begin{lemma} \label{l: proper quotients of B}
    Let $G$ be a group, and let $S \subseteq G$ such that $1 \notin S$. Then, the ideal of $\mathcal{B}$ generated by $\{ e_{s} : s \in S \}$ is different from $\mathcal{B}$.
\end{lemma}
\begin{proof}
    Note that the ideal generated by $\{ e_{s} : s \in G \}$ has the basis $\{ e_{s}e_{g_1} \ldots e_{g_m} : s \in S, m \in \mathbb{N}, \text{ and } g_i \in G \}$ as a $K$-module. Finally, observe that such a set is linearly independent to $1$.
\end{proof}

\begin{lemma} \label{l: zeros of KparG}
    Let $G$ be an infinite group and $z \in K_{par}G$. If there exists $S \subseteq G$ infinite such that $ze_{g} = 0$ for all $g \in S$ then $z = 0$.
\end{lemma}
\begin{proof}
    Suppose that we have already proven the lemma for the case $z \in \mathcal{B}$. Let $x \in K_{par}G$, since $K_{par}G = \oplus_{h \in G} [h] \mathcal{B}$, there exists $\{ w_{h} \}_{h \in G} \subseteq \mathcal{B}$ such that $x = \sum_{h \in G} [h]w_{h}$. Then, $z e_s = 0$ if, and only if, $[h]w_{h}e_{s} = 0$ for all $h \in G$. Note that $[h]w_h e_s = 0$ if, and only if, $e_{h^{-1}}w_{h} e_s = 0$. Suppose that there exists $S$ infinite such that $ze_s = 0$ for all $s \in G$. Then, $e_{h^{-1}}w_h e_s = 0$ for all $s \in S$ and $h \in G$, thus by hypotheses $e_{h^{-1}}w_h = 0$ for all $h \in G$. Therefore, $z = 0$.

    Thus, it is enough to prove the lemma for the case where $z \in \mathcal{B}$. Recall that $\mathcal{B}$ as $K$-module has basis $E(\mathcal{S}(G))=\{ e_{g_0} e_{g_1} \ldots e_{g_n} : n \in \mathbb{N}, \, g_i \in G \}$. Moreover, set $\mathcal{P} := \{ U \subseteq G : 1 \in U, \, |U| < \infty \}$, then the function
    \begin{align*}
        \xi: \mathcal{P} &\to E(\mathcal{S}(G)) \\
        U & \to \prod_{g \in U}e_g
    \end{align*}
    is bijective.
    If $z = 0$, then the proof is trivial. Let $z \in \mathcal{B} \setminus \{ 0 \}$,then there exists $\{ a_i \}_{i = 0}^{n} \subseteq K \setminus \{ 0 \}$ and $\{ U_i \}_{i = 0}^{n} \subseteq \mathcal{P}$ such that
    \[
        z = \sum_{i = 0}^{n} a_i \xi(U_i),
    \]
    and the set $\{ \xi(U_i) : 0 \leq i  \leq n\}$ is linear independent.
    Furthermore, observe that the set $\{ \xi(U_i \cup \{ g \} ) : 0 \leq i \leq n\}$ is linear independent for all $g \in G \setminus \cup^{n}_{i = 0} U_i$. Therefore,
    \[
       z e_g = \sum_{i = 0}^{n} a_i \xi(U_i \cup \{ g \}) \neq 0 \text{ for all } g \in G \setminus \cup_{i=0}^{n} U_i.
    \]
    Thus, we conclude that if $z e_s = 0$, then $s \in \cup_{i=0}^{n} U_i$. Now suppose that there exists $S \subseteq G$ infinite such that $ze_s = 0$ for all $s \in S$. Then, $S \subseteq \cup_{i = 0}^{n} U_i$, but this is a contradiction since the $\cup_{i = 0}^{n} U_i$ is finite. Thus, if $z \neq 0$ it only can be annihilated by finitely many $e_s$'s.
\end{proof}

Consider the surjective map of right $K_{par}G$-modules 
\begin{equation}
    \delta: KG \otimes K_{par}G \to \mathcal{N}
    \label{eq: delta map}
\end{equation}
such that $\delta(g \otimes 1) := g \otimes \nu_{g^{-1}}$. Observe that $\delta(g \otimes \nu_{g}) = g \otimes \nu_{g}$. Indeed,
\begin{align*}
    \delta(g \otimes \nu_{g}) 
    &= \delta(g \otimes 1) \triangleleft \nu_{g} \\ 
    &=  (g \otimes \nu_{g}) \triangleleft 1  - (g \otimes \nu_{g}) \triangleleft e_{g} \\
    &=  g \otimes \nu_{g}  - g \otimes \nu_{g} e_{g} \\
    &=  g \otimes \nu_{g}.
\end{align*}

\begin{proposition} \label{p: uncountable group implies no projective}
    Let $G$ be an uncountable infinite group, then $KG \otimes_{G_{par}} K_{par}G$ is not projective as right $K_{par}G$-module.
\end{proposition}
\begin{proof}
    Let $KG \otimes K_{par}G$ be the right $K_{par}G$-module where $(h \otimes z) \triangleleft [g] := h \otimes z[g]$, note that with that structure we have that $KG \otimes K_{par}G$ is a free right $K_{par}G$-module with basis $\{ g \otimes 1 : g \in G \}$.
    Let $\delta$ be the map \eqref{eq: delta map} and let $\pi_g : KG \otimes K_{par}G \to K_{par}G$ be the maps such that 
    \[
        \pi_g(h \otimes z)=\left\{\begin{matrix}
            z & \text{ if } h = g, \\ 
            0 & \text{ otherwise} 
        \end{matrix}\right.
    \]
    for all $h \in G$ and $z \in K_{par}G$. Suppose that $\mathcal{N}$ is projective, then there exists a morphism of right $K_{par}G$-modules $\phi: \mathcal{N} \to KG \otimes K_{par}G$ such that $\delta \circ \phi = 1_{\mathcal{N}}$. Under the above hypotheses, we divide the rest of the proof into four steps:    
    \begin{enumerate}
        \item[\textbf{Step 0}] We note that the set $A_t := \{ s \in G : \pi_s(\phi(t \otimes \nu_{t^{-1}})) \neq 0 \}$ is finite for all $t \in G$. 
        \item[\textbf{Step 1}] We show that $x_g := \pi_g(\phi(g \otimes \nu_{g^{-1}})) \neq 0$ for all $g \in G$.
        \item[\textbf{Step 2}] We prove that the set $Z^{g} := \{ t \in G : \pi_g(\phi(t \otimes \nu_{t^{-1}})) = 0 \}$ is finite for all $g \in G$.
        \item[\textbf{Step 3}] We show that for all $S \subseteq G$ infinite and countable, there exists $t \in G$ such that $ S \subseteq A_{t}$, which leads to a contradiction.
    \end{enumerate}

    \textbf{Step 0}. Observe that $\phi(t \otimes \nu_{t^{-1}}) = \sum_{g \in G} g \otimes \pi_g(\phi(t \otimes \nu_{t^{-1}}))$. Therefore, $A_t$ is finite.

    \textbf{Step 1}. Let $g \in G$, recall that $K_{par}G = \bigoplus_{h \in G} [h] \mathcal{B}$, then there exists $\{ u_{s,h} \}_{s, h \in G} \subseteq \mathcal{B}$ such that and $\phi(g \otimes \nu_{g^{-1}}) = \sum_{s, h \in G} s \otimes [h] u_{s,h}$. Then,
    \begin{align*}
        g \otimes \nu_{g^{-1}} 
        &= \delta(\phi(g \otimes \nu_{g^{-1}})) 
        = \delta \left( \sum_{s, h \in G} s \otimes [h] u_{s, h} \right) \\
        &= \sum_{s, h \in G} (s \otimes \nu_{s^{-1}}) \triangleleft [h] u_{s, h}
        = \sum_{s, h \in G} sh \otimes \nu_{h^{-1}s^{-1}} e_{h^{-1}} u_{s, h}.
    \end{align*}
    Thus, $\nu_{g^{-1}} = \sum_{h \in G}\nu_{g^{-1}}e_{h^{-1}} u_{gh^{-1}, h}$. Therefore,
    \[
        \nu_{g^{-1}}u_{g, 1} = \nu_{g^{-1}} - \sum_{\substack{h \in G \\ h \neq 1}} \nu_{g^{-1}}u_{gh^{-1}, h}e_{h^{-1}} = 1 - e_{g^{-1}} - \sum_{\substack{h \in G \\ h \neq 1}} \nu_{g^{-1}}u_{gh^{-1}, h}e_{h^{-1}}.
    \]
    Observe that $e_{g^{-1}} + \sum_{\substack{s \in G \\ s \neq 1}} \nu_{g^{-1}}u_{gh^{-1}, h}e_{h^{-1}}$ is in the ideal generated by the set $\{ e_{g^{-1}} : g \in G \setminus \{ 1 \} \}$, hence by Lemma~\ref{l: proper quotients of B} we conclude that $\nu_{g^{-1}}u_{g, 1} \neq 0$, and therefore $u_{g, 1} \neq 0$. Note that
    \[
       x_g :=\pi_g \phi( g \otimes \nu_{g^{-1}}) = \sum_{h \in G} [h]u_{g,h},
    \]
    thus $\pi_g \phi( g \otimes \nu_{g^{-1}}) = 0$ if, and only if, $[h]u_{g,h}=0$ for all $h \in G$, but $[1]u_{g,1} = u_{g, 1} \neq 0$. Therefore, $x_g \neq 0$ for all $g \in G$. 

    \textbf{Step 2}. Since $\phi$ is a morphism of right $K_{par}G$-modules we get
    \begin{equation} \label{eq: xg xt equality}
        \phi(t \otimes \nu_{t^{-1}})[t^{-1}g] 
        = \phi((t \otimes \nu_{t^{-1}}) \triangleleft [t^{-1}g]) 
        = \phi(g \otimes \nu_{g^{-1}}e_{g^{-1}t}) 
        = \phi(g \otimes \nu_{g^{-1}}) e_{g^{-1}t},
    \end{equation}
    for all $g, t \in G$. Let $g \in G$, define $Z^{g}:= \{ t \in G: \pi_g(\phi(t \otimes \nu_{t^{-1}}))=0\}$. By equation \eqref{eq: xg xt equality} we have
    \[
        x_g e_{g^{-1}t}  = \pi_g (\phi(g \otimes \nu_{g^{-1}})) e_{g^{-1}t} = \pi_g(\phi(t \otimes \nu_{t^{-1}}))[t^{-1}g] = 0 \text{ for all } t \in Z^{g}.
    \]
    By \textit{Step 1} we know that $x_g \neq 0$, then, by Lemma~\ref{l: zeros of KparG}, we conclude that $Z^{g}$ is finite for all $g \in G$. 

    \textbf{Step 3}. Let $S$ be an infinite countable subset of $G$. Thus, by \textit{Step 2}, $T=G \setminus \cup_{s \in S} Z^s$ is an infinite uncountable set. Let $t \in T$, then $\pi_s(\phi(t \otimes \nu_{t^{-1}})) \neq 0$ for all $s \in S$, thus $S \subseteq A_t$, but by \textit{Step 0} the set $A_t$ is finite. Therefore, $KG \otimes \mathcal{B}$ cannot be projective as right $K_{par}G$-module.
\end{proof}

\begin{lemma} \label{l: existence of map on N}
    Let $G = \{ g_n \}_{n \in \mathbb{N}}$ be a countable group, where $g_0 = 1$. If there exists $\{ x_n \}_{n \geq 1} \subseteq KG \otimes K_{par}G$ such that:
    \begin{enumerate}[(i)]
        \item $\delta(x_n) = g_n \otimes \nu_{g_n^{-1}}$,
        \item for all $n \in \mathbb{N}$, $x_n[g_{n}^{-1}g_r] = x_r e_{g_r^{-1} g_n}$ for all $r \leq n$,
    \end{enumerate}
    Then, the map $\phi: \mathcal{N} \to KG \otimes K_{par}G$ given by $\phi(g_n \otimes \nu_{g_n^{-1}}u) := x_n \nu_{g_n^{-1}}u$ is a well-defined morphism of right $K_{par}G$-modules such that $\delta \phi = 1_{\mathcal{N}}$.
\end{lemma}
\begin{proof}
    It is clear that $\phi$ is a well-defined $K$-linear map since $
    \{ g_n \otimes \nu_{g_n^{-1}}u : n \geq 1, \, u \in E(\mathcal{S}(G)) \}$ is a basis of $\mathcal{N}$ as $K$-module. Note that in $(ii)$ if we multiply by $[g_r^{-1} g_n]$ to the right we obtain
    \[
        x_ne_{g_{n}^{-1}g_r} = x_n[g_{n}^{-1}g_r][g_r^{-1} g_n] = x_r e_{g_r^{-1} g_n}[g_r^{-1} g_n] = x_r [g_r^{-1} g_n],
    \]
    whence we conclude that $(ii)$ is valid not only for $r \leq n$ but for all $r,n \in \mathbb{N}$. Let $u \in \mathcal{B}$ and $t \in G$, observe that for any $n \geq 0$ there exists $r \geq 0$ such that $t = g_{n}^{-1} g_r$, then
    \begin{align*}
        \phi\big( (g_n \otimes \nu_{g_n^{-1}} u) \triangleleft [t]\big)
        &= \phi\big( (g_n \otimes \nu_{g_n^{-1}} u) \triangleleft [g_n^{-1}g_r]\big) \\
        &= \phi \big( g_r \otimes \nu_{g_r^{-1}} e_{g_r^{-1}g_n} (u \triangleleft [g_n^{-1}g_r]) \big) \\
        &= x_{r}e_{g_r^{-1}g_n}\nu_{g_r^{-1}}(u \triangleleft [g_n^{-1}g_r]) \\
        \text{by } (ii) \,  &= x_{n}[g_n^{-1}g_r]\nu_{g_r^{-1}}(u \triangleleft [g_n^{-1}g_r]) \\
        &= x_{n}\nu_{g_n^{-1}}[g_n^{-1}g_r](u \triangleleft [g_n^{-1}g_r]) \\
        &= x_{n}\nu_{g_n^{-1}}[g_n^{-1}g_r]([g_r^{-1}g_n]u [g_n^{-1}g_r]) \\
        &= x_{n}\nu_{g_n^{-1}}e_{g_n^{-1}g_r}u [g_n^{-1}g_r] \\
        &= x_{n}\nu_{g_n^{-1}}u [g_n^{-1}g_r] \\
        &= \phi(g_n \otimes \nu_{g_n^{-1}}u) [g_n^{-1}g_r] \\
        &= \phi(g_n \otimes \nu_{g_n^{-1}}u) [t].
    \end{align*}
    Then, $\phi$ is a well-defined map of right $K_{par}G$-modules. Finally, note that
    \[
        \delta\big( \phi(g_n \otimes \nu_{g_n^{-1}}) \big) = \delta( x_n \nu_{g_n^{-1}}) = \delta(x_n) \nu_{g_n^{-1}} = g_n \otimes \nu_{g_n^{-1}}.
    \]
    Hence, $\delta \circ \phi = 1_\mathcal{N}$.
\end{proof}

\begin{proposition} \label{p: countable group implies KG B is projective}
    Let $G$ be an infinite countable group, then $KG \otimes \mathcal{B}$ is projective as right $K_{par}G$-module.
\end{proposition}
\begin{proof}
    Let $\{ g_n \}_{n \in \mathbb{N}}$ be an enumeration of $G$ such that $g_0 = 1$. Define $x_1 := g_1 \otimes \nu_{g_1^{-1}}$, and we define $\{ x_n \}_{n \geq 1} \subseteq KG \otimes K_{par}G$ recursively as follows: suppose that we have already defined $x_r$ for all $r < n$. Set
    \[
       x_{n,1} := g_n \otimes \nu_{g_n^{-1}} + x_1 [g_1^{-1}g_n] - g_n \otimes \nu_{g_n^{-1}}e_{g_n^{-1}g_1},
    \]
    and recursively define:
    \[
       x_{n, r} := x_r [g_r^{-1}g_n] + x_{n,r-1}\nu_{g_n^{-1}g_r} \text{for all } r \in  \{ 2, 3, \ldots n-1 \}.
    \]
    Set \( x_n := x_{n, n-1} \). We affirm that the set \( \{ x_n \}_{n \geq 1} \) satisfies the hypotheses of Lemma~\ref{l: existence of map on N}. To prove this affirmation, we need to employ nested induction arguments. For the sake of clarity, we will explicitly write the begining and end of each induction argument.
    
    $\blacktriangleleft$ \textbf{Induction A}. Note that for $n = 1$ we have that $x_1$ trivially satisfies conditions $(i)$ and $(ii)$ of Lemma~\ref{l: existence of map on N}, thus establishing the base of our induction argument. Let $n > 1$, suppose that we already proved that $x_r$ satisfies $(i)$ and $(ii)$ of Lemma~\ref{l: existence of map on N} for all $r < n$. Observe that
    \begin{align*}
        x_{n, 1} [g_n^{-1} g_1]  
        &= \big(g_n \otimes \nu_{g_n^{-1}} + x_1[g_1^{-1} g_n] - g_n \otimes \nu_{g_n^{-1}}e_{g_n^{-1}g_1}\big)[g_n^{-1} g_1] \\ 
        &= g_n \otimes \nu_{g_n^{-1}}[g_n^{-1} g_1] + g_1 \otimes \nu_{g_1^{-1}}[g_1^{-1} g_n][g_n^{-1} g_1] - g_n \otimes \nu_{g_n^{-1}}e_{g_n^{-1}g_1}[g_n^{-1} g_1] \\ 
        &= g_n \otimes \nu_{g_n^{-1}}[g_n^{-1} g_1] + g_1 \otimes \nu_{g_1^{-1}}e_{g_1^{-1}g_n} - g_n \otimes \nu_{g_n^{-1}}[g_n^{-1} g_1] \\ 
        &= g_1 \otimes \nu_{g_1^{-1}}e_{g_1^{-1}g_n} = x_1 e_{g_1^{-1}g_n}.
    \end{align*}
    For $1 < r < n$, observe that 
    \[
        x_{n, r} [g_n^{-1} g_r]  
        = x_r [g_r^{-1}g_n][g_n^{-1} g_r] + x_{n,r-1}\nu_{g_n^{-1}g_r}[g_n^{-1} g_r] 
        = x_r e_{g_r^{-1}g_n} + 0 
        = x_r e_{g_r^{-1}g_n}.
    \]
    Then,
    \begin{equation}\label{eq: G countable formula 1}
        x_{n,r} [g_n^{-1} g_r] = x_r e_{g_r^{-1} g_n} \text{ for all } r \in \{ 1, 2, \ldots n-1 \}.
    \end{equation}
    We want to prove that
    \begin{equation} \label{eq: G countable formula 2}
        x_{n, r+m} [g_n^{-1}g_r] = x_{r} e_{g_r^{-1} g_n},
    \end{equation}
    for all $r \in \{ 1, 2, \ldots, n-1 \}$ and $ m \in \{0, 1, 2, \dots, n - 1 - r \}$. 

    We will prove equation \eqref{eq: G countable formula 2} by induction over $m$. 

    $\blacktriangleleft \blacktriangleleft$ \textbf{Induction A1}. Fix $r \in \{ 1, 2, \ldots, n-1 \}$. Note that \eqref{eq: G countable formula 1} is the base of the induction. For the general case take $m \in \{ 1, 2, \dots, n - 1 - r \}$ and suppose that $x_{n, r+m-1}[g_n^{-1} g_r] = x_r e_{g_r^{-1} g_n}$. Then,
    \begin{align*}
        x_{n, r+m} [g_n^{-1}g_r] 
        &= x_{r+m}[g_{r+m}^{-1}g_n] [g_n^{-1}g_r]+ x_{n, r+m-1} \nu_{g_n^{-1} g_{r+m}}[g_n^{-1}g_r] \\
        &= x_{r+m}[g_{r+m}^{-1}g_r] e_{g_r^{-1} g_n} + x_{n, r+m-1} [g_n^{-1}g_r] \nu_{g_r^{-1} g_{r+m}} \\
        (\flat) &= x_{r}e_{g_r^{-1}g_{r+m}} e_{g_r^{-1} g_n} + x_{r} e_{g_r^{-1}g_n} \nu_{g_r^{-1} g_{r+m}} \\
        &= x_{r} e_{g_r^{-1}g_n}(e_{g_r^{-1}g_{r+m}}  + \nu_{g_r^{-1} g_{r+m}})\\
        &= x_{r} e_{g_r^{-1} g_n},
    \end{align*}
    where the equality $(\flat)$ holds since $x_{r+m}[g_{r+m}^{-1}g_r] = x_{r}e_{g_r^{-1}g_{r+m}}$ by the hypotheses of the \textbf{Induction A},  
    and $x_{n, r+m-1} [g_n^{-1}g_r] = x_{r} e_{g_r^{-1}g_n}$ by the induction hypotheses of the \textbf{Induction A1}. \textbf{End of Induction A1}~$\blacktriangleright \blacktriangleright$. Therefore, equation~\eqref{eq: G countable formula 2} holds.

    In particular, when $m = n - 1 - r$ in equation \eqref{eq: G countable formula 2} we obtain 
    \[
       x_n [g_n^{-1} g_r] = x_{n, n-1} [g_n^{-1} g_r] = x_r e_{g_r^{-1}g_n}
    \]
    for all $r \leq n$. Thus, the set $\{ x_r \}_{r \leq n}$ satisfies the condition $(ii)$ of Lemma~\ref{l: existence of map on N}. 

    To verify that $\{ x_r \}_{r \leq n}$ also satisfies the condition $(i)$ of Lemma~\ref{l: existence of map on N} we have to perform another induction argument.

     $\blacktriangleleft \blacktriangleleft$~\textbf{Induction A2}. Observer that
     \begin{align*}
       \delta(x_{n,1}) 
         &= \delta( g_n \otimes \nu_{g_n^{-1}} + x_1[g_1^{-1} g_n] - g_n \otimes \nu_{g_n^{-1}}e_{g_n^{-1}g_1}) \\
         &= \delta( g_n \otimes \nu_{g_n^{-1}} + (g_1 \otimes \nu_{g_1^{-1}})[g_1^{-1} g_n] - g_n \otimes \nu_{g_n^{-1}}e_{g_n^{-1}g_1}) \\
         &= g_n \otimes \nu_{g_n^{-1}} + (g_1 \otimes \nu_{g_1^{-1}}) \triangleleft [g_1^{-1} g_n] - g_n \otimes \nu_{g_n^{-1}}e_{g_n^{-1}g_1} \\
         &= g_n \otimes \nu_{g_n^{-1}} + g_n \otimes [g_n^{-1} g_1]\nu_{g_1^{-1}}[g_1^{-1} g_n] - g_n \otimes \nu_{g_n^{-1}}e_{g_n^{-1}g_1} \\
         &= g_n \otimes \nu_{g_n^{-1}} + g_n \otimes \nu_{g_n^{-1}}e_{g_n^{-1}g_1} - g_n \otimes \nu_{g_n^{-1}}e_{g_n^{-1}g_1} \\
         &= g_n \otimes \nu_{g_n^{-1}}.
     \end{align*}
     This give us the base of the induction. Let $k \in \{ 2, \ldots, n -1 \}$, suppose that we already have proven that $\delta(x_{n, k-1}) = g_n \otimes \nu_{g_n^{-1}}$, then
    \begin{align*}
        \delta(x_{n, k}) 
        &= \delta\big(x_{k}[g_{k}^{-1}g_n] + x_{n, k-1}\nu_{g_n^{-1}g_{k}} \big) \\
        (\flat) &= (g_{k} \otimes \nu_{g_{k}^{-1}}) \triangleleft [g_{k}^{-1}g_n] + g_n \otimes \nu_{g_n^{-1}}\nu_{g_n^{-1}g_{k}} \\
        &= g_{n} \otimes \nu_{g_{n}^{-1}} e_{g_n^{-1}g_{k}} + g_n \otimes \nu_{g_n^{-1}}\nu_{g_n^{-1}g_{k}} \\
        &= g_{n} \otimes \nu_{g_{n}^{-1}},
    \end{align*}
    where $(\flat)$ holds since $\delta(x_k) = g_k \otimes \nu_{g_k^{-1}}$ by the hypotheses of \textbf{Induction A} and $\delta(x_{n, k-1}) = g_n \otimes \nu_{g_n^{-1}}$ by the hypotheses of \textbf{Induction A2}.
    Therefore, $\delta(x_n)  = \delta(x_{n, n-1}) = g_n \otimes \nu_{g_n^{-1}}$. Hence, $\{ x_r \}_{r \leq n}$ satisfies $(i)$ of Lemma~\ref{l: existence of map on N}. \textbf{End of Induction A2}~$\blacktriangleright \blacktriangleright$. 

    From \textbf{Induction A1} and \textbf{Induction A2} we conclude that the set \( \{ x_r \}_{r \leq n} \) satisfies the hypotheses of Lemma~\ref{l: existence of map on N}.
    \textbf{End of Induction A}~$\blacktriangleright$.

    By \textbf{Induction A} we know that $\{ x_n \}_{n \geq 1}$, satisfies the hypotheses of Lemma~\ref{l: existence of map on N}. Therefore, there exists a map of right $K_{par}G$-modules $\phi: \mathcal{N} \to KG \otimes K_{par}G$ such that $\delta \phi = 1_{\mathcal{N}}$. Hence, by Lemma~\ref{l: N is projective iff} we conclude that $K_{par}G \otimes \mathcal{B}$ is projective.
\end{proof}

\begin{remark} \label{r: finite group implies KG B projective}
    If $G$ is a finite group, then the construction made in Proposition~\ref{p: countable group implies KG B is projective} also holds, implying that $KG \otimes \mathcal{B}$ is projective as a right $K_{par}G$-module. An alternative way to verify this is to consider the following exact sequence:
    \[
        0 \to \mathcal{K} \to KG \otimes K_{par}G \to KG \otimes_{G_{par}} K_{par}G \to 0,
    \]
    where $\mathcal{K}$ is the module defined in equation \eqref{eq: kernel globalization partial tensor product}. Note that $\mathcal{K}$ is finitely generated, as it is generated as a right $K_{par}G$-module by the finite set $\{ gh \otimes e_{h^{-1}} - g \otimes [h] : g, h \in G \}$. Thus, $KG \otimes_{G_{par}} K_{par}G$ is a finitely presented right $K_{par}G$-module. By Corollary~\ref{c: the globalizer module is KparG flat}, we know that $KG \otimes_{G_{par}} K_{par}G$ is flat as a right $K_{par}G$-module. Therefore, by \cite[Theorem 3.56]{RotmanAnInToHoAl}, we conclude that $KG \otimes_{G_{par}} K_{par}G$ is projective.
\end{remark}

\begin{proposition} \label{p: dual modules collapses}
    Let $M$ be a left $K_{par}G$-module, suppose that $K$ is self-injective. Then, $\operatorname{Ext}^{n}_{K_{par}G}(KG \otimes_{G_{par}} K_{par}G, M^{*})=0$ for all $n \geq 1$.
\end{proposition}
\begin{proof}
    Recall that $\hom_K(M, K)$ is a right $K_{par}G$-module with the action determined by $(f \triangleleft [g])(m):= f([g] \cdot m)$. Let $P_\bullet \to M$ be a projective resolution of left $K_{par}G$-modules. Then, by \cite[Lemma 3.5]{Lam1998LecturesOM}, we conclude that $\hom_K(M, K) \to \hom_K(P_\bullet, K)$ is an injective resolution of right $K_{par}G$-modules. Then, for $n \geq 1$, we have
    \begin{align*}
        \operatorname{Ext}^{n}_{K_{par}G}( KG \otimes_{G_{par}}K_{par}G, M^{*})
        &\cong H^{n}\big(\hom_{K_{par}G}\big( KG \otimes_{G_{par}} K_{par}G, \hom_K(P_\bullet, K) \big)  \big)\\
        (\flat)&\cong H^{n}\big(\hom_{K}\big( KG \otimes_{G_{par}} K_{par}G \otimes_{K_{par}G} P_\bullet, K \big) \big) \\
        (\flat \flat) &\cong H^{n}\big(\hom_{K}( KG \otimes_{G_{par}} P_\bullet, K) \big)=0. 
    \end{align*}
    The equality $(\flat)$ holds because of the well-known tensor-hom adjunction, the equality $( \flat \flat)$ holds by Proposition~\ref{p: partial tensor product is associative}. Finally, the last equality holds since the functors $\hom_K(-, K)$ and $KG \otimes_{G_{par}} -$ are exact.
\end{proof}

Combining Proposition~\ref{p: cohomological globalization spectral sequence}, Proposition~\ref{p: countable group implies KG B is projective}, Remark~\ref{r: finite group implies KG B projective}, and Proposition~\ref{p: dual modules collapses}, we obtain the following theorem:

\begin{theorem} \label{t: cohomology main theorem}
    Let $M$ be a right $K_{par}G$-module. Then, there exists a cohomology spectral sequence
    \[
        E_{2}^{p,q} := H^p\big(G, \operatorname{Ext}^{q}_{K_{par}G}( \Lambda(K_{par}G) , M) \big) \Rightarrow H^{p+q}_{par}(G, M).
    \]
    If $G$ is finite or countable, or if $K$ is self-injective and $M$ is the dual module of a left $K_{par}G$-module, then the spectral sequence collapses and gives rise to an isomorphism
    \[
        H^{n}_{par}(G, M) \cong H^{n}(G, \hom_{K_{par}G}( \Lambda(K_{par}G), M)).
    \]
\end{theorem}

\begin{remark}
    Note that by Proposition~\ref{p: uncountable group implies no projective} the spectral sequence in Theorem~\ref{t: cohomology main theorem} may not collapse when $G$ is an infinite uncountable group.
\end{remark}

\begin{proposition} \label{p: Omega of a KG module is trivial}
    Let $M$ be a right $KG$-module. Then,
    \[
        \operatorname{Ext}^{n}_{K_{par}G}(\Lambda(K_{par}G), M) =  \left\{\begin{matrix}
           M, & \text{ if }\,\, n = 0\\
           0, & \text{ otherwise.}
        \end{matrix}\right.
    \]
\end{proposition}
\begin{proof}
    Let $P_{\bullet} \to KG \otimes_{G_{par}} K_{par}G$ be a projective resolution of $KG \otimes_{G_{par}} K_{par}G$ of right $K_{par}G$-modules.
    Then, $P_{\bullet} \otimes_{K_{par}G}KG$ is a projective resolution of right $KG$-modules of $KG$. Indeed, by \cite[Theorem 6.6]{MMAMDDHKPartialHomology} we have that $KG$ is flat as a $K_{par}G$-module.
    Therefore, by Item $(i)$ of \cite[Lemma 10.69]{RotmanAnInToHoAl} $P_{\bullet} \otimes_{K_{par}G}KG$ is a projective resolution of right $KG$-modules of $(KG \otimes_{G_{par}} K_{par}G) \otimes_{K_{par}G} KG$.
    Moreover, by Proposition~\ref{p: partial tensor product is associative} we know that $(KG \otimes_{G_{par}} K_{par}G) \otimes_{K_{par}G} KG \cong KG$.
    Now observe that
    \begin{align*}
        \operatorname{Ext}^{n}_{K_{par}G}(KG \otimes_{G_{par}}K_{par}G, M)
        &\cong H^{n}(\hom_{K_{par}G}(P_{\bullet}, M)) \\
        &\cong H^{n}(\hom_{K_{par}G}(P_{\bullet}, \hom_{KG}(KG,M))) \\
        &\cong H^{n}(\hom_{KG}(P_{\bullet} \otimes_{K_{par}G} KG, M)) \\
        &\cong \operatorname{Ext}^{n}_{KG}(KG, M)
    \end{align*}
    whence we obtain the desired conclusion.
\end{proof}

Note that if $M$ is a right $KG$-module, Proposition~\ref{p: Omega of a KG module is trivial} implies that Theorem~\ref{t: cohomology main theorem} gives us the trivial isomorphism $H^{n}(G, M) \cong H^{n}(G,M)$. 

We recall from \cite{MMAMDDHKPartialHomology} that the \textbf{partial cohomological dimension} $cd^{par}_{K}(G)$ of $G$ (over $K$) is defined as the projective dimension of $\mathcal{B}$ as a $K_{par}G$-module. Therefore,
\[
    cd^{par}_{K}(G):= \operatorname{max}\left\{ n \in \mathbb{N} : H^{n}_{par}(G, M) \neq 0 \text{ for some right } K_{par}G\text{-module } M\right\}.
\]
By \cite[Corollary 6.7]{MMAMDDHKPartialHomology} we know that $cd_{K}(G) \leq cd_{K}^{par}(G)$. A direct consequence of Theorem~\ref{t: cohomology main theorem} is the following:
\begin{corollary} \label{c: partial cohomological dimension of G}
    If $G$ is finite or countable, then $cd^{par}_{K}(G) = cd_{K}(G)$. For the case where $G$ is infinite uncountable we obtain the following inequality $cd^{par}_{K}(G) \leq cd_{K}(G) + pd_{K_{par}G}(KG \otimes_{G_{par}} K_{par}G)$.
\end{corollary}

Observe that Corollary~\ref{c: partial cohomological dimension of G} provides a positive answer to \textbf{Conjecture D} in \cite{MMAMDDHKPartialHomology} when $G$ is countable.

\section*{Acknowledgments}
The research was supported by the Fundação de Amparo à Pesquisa do Estado de São Paulo (FAPESP) under grant number 2022/12963-7.

\bibliographystyle{abbrv}
\bibliography{azu.bib}

\end{document}